\documentclass[10pt,twoside,a4paper,final]{amsart}

\usepackage[latin1]{inputenc}

\usepackage[centertags]{amsmath}
\usepackage{amsthm}
\usepackage{amssymb}
\usepackage{amsfonts,a4}
\usepackage{amsxtra}
\usepackage{amscd}
\usepackage{mathabx} 
\usepackage{esint}

\usepackage[english]{babel}

\usepackage{algorithm}
\usepackage{algorithmic}

\usepackage{color}
\usepackage{todonotes}
\usepackage[mathscr]{eucal}

\usepackage{bbm}
\usepackage{enumerate}
\usepackage{enumitem}
\setlist{leftmargin=5ex}

\usepackage{graphicx}
\graphicspath{{./plots/}}

\usepackage{xspace} %
\usepackage[notcite,notref]{showkeys}

\usepackage[scrtime]{prelim2e}

\newtheorem{thm}{Theorem}
\newtheorem{cor}[thm]{Corollary}
\newtheorem{lem}[thm]{Lemma}

\theoremstyle{definition}

\newtheorem{rem}[thm]{Remark}
\theoremstyle{definition}

\numberwithin{equation}{section}

\newcommand{\Ao}{\ensuremath{\mathcal{A}}}
\newcommand{\Aon}[1][n]{\ensuremath{\Ao_{n}}}

\newcommand{\abs}[1]{\ensuremath{\left|#1\right|}}

\newcommand{\Cgeq}{\ensuremath{\gtrsim}}

\newcommand{\Cleq}{\ensuremath{\lesssim}}

\newcommand{\definedas}{\mathrel{:=}}
\newcommand{\dual}[2]{\ensuremath{\left\langle #1,\,#2\right\rangle}}

\newcommand{\dx}{\ensuremath{\,\mathrm{d}x\xspace}}
\newcommand{\ds}{\ensuremath{\,\mathrm{d}s\xspace}}

\DeclareMathOperator{\diam}{diam}
\DeclareMathOperator{\divo}{div}

\newcommand{\elm}{\ensuremath{K}\xspace}

\newcommand{\est}{\mathcal{E}}
\newcommand{\estR}{\est_{\mathrm{R}}}

\newcommand{\grid}{\mathcal{M}}
\newcommand{\gridn}[1][n]{\grid_{#1}}

\newcommand{\helm}[1][\ell]{\ensuremath{h_{\elm}}}

\newcommand{\ie}{\hbox{i.\,e.},\xspace}

\DeclareMathAlphabet{\lf}{OT1}{pzc}{m}{it}
\newcommand{\lff}{f}
\newcommand{\lfg}{\ell}

\newcommand{\N}{\ensuremath{\mathbb{N}}}

\newcommand{\vertices}{\ensuremath{\mathcal{V}}}
\newcommand{\normal}{\ensuremath{\vec{n}}}

\newcommand{\norm}[2][\Omega]{\ensuremath{\|#2\|_{#1}}}

\DeclareMathOperator{\osc}{osc}

\newcommand{\OFDG}[1][\grid]{\underline{\mathbb{D}}(#1)}

\providecommand{\Poincare}{{Poincar{\'e}}\xspace}
\renewcommand{\P}{\ensuremath{\mathbb{P}}}

\renewcommand{\paragraph}[1]{\noindent\raisebox{0pt}[10pt][0pt]{\textbf{#1.}}}

\newcommand{\ProG}[1][\grid]{\ensuremath{\mathcal{P}_{#1}}}

\newcommand{\RTN}{\ensuremath{\text{RTN}}}
\newcommand{\R}{\ensuremath{\mathbb{R}}}
\DeclareMathOperator{\Res}{Res}

\newcommand{\scp}[3][\Omega]{\ensuremath{\left\langle #2,\,#3\right\rangle}_{#1}}

\newcommand{\side}{\ensuremath{F}\xspace}
\newcommand{\sides}{\mathcal{F}}

\DeclareMathOperator{\supp}{supp}

\newcommand{\TProG}[1][\grid]{\ensuremath{\Tilde{\mathcal{P}}_{#1}}}

\newcommand{\U}{\ensuremath{\mathbb{U}}}

\renewcommand{\vec}[1]{\ensuremath{\boldsymbol{#1}}}

\newcommand{\vol}[1]{\abs{#1}}
\newcommand{\V}{\ensuremath{\mathbb{V}}}

\newcommand{\VG}[1][\grid]{\V(#1)}
\newcommand{\VoG}[1][\grid]{\V_0(#1)}

\newcommand{\W}{\ensuremath{\mathbb{W}}}

\begin{document}


%
%
%
\title
{Oscillation in a~posteriori error estimation}

\author[Ch.~Kreuzer]{Christian Kreuzer}
\address{Christian Kreuzer,
 Fakult\"at f\"ur Mathematik,
 Technische Universit\"at Dortmund, 
 Vogelpothsweg 87, D-44227 Dortmund, Germany 
 }%
\urladdr{http://www.mathematik.tu-dortmund.de/lsx/cms/de/mitarbeiter/ckreuzer.html}
\email{christan.kreuzer@tu-dortmund.de}

\author[A.~Veeser]{Andreas~Veeser}
\address{Andreas~Veeser,  Dipartimento di Matematica, Universit\`a
  degli Studi di Milano, Via C. Saldini 50,
 I-20133 Milano, Italy
 }%
\urladdr{http://mat.unimi.it/users/veeser/}
\email{andreas.veeser@unimi.it}



\begin{abstract}
In a~posteriori error analysis, the relationship between error and estimator is usually spoiled by so-called oscillation terms, which cannot be bounded by the error. In order to remedy, we devise a new approach where the oscillation has the following two properties. First, it is dominated by the error, irrespective of mesh fineness and the regularity of data and the exact solution. Second, it captures in terms of data the part of the residual that, in general, cannot be quantified with finite information. The new twist in our approach is a locally stable projection onto discretized residuals. 
\end{abstract}

\date{\small\today}

\maketitle

\section{Introduction}
\label{s:introduction}
%
%
Finite element methods are a successful and well-established technique
for the solution of partial differential equations.  A key tool for
the quality assessment of a given finite element approximation and the
application of adaptive techniques are so-called \emph{a~posteriori
  error estimators}.  These are functionals that are computable in
terms of data and the finite element approximation and aim at
quantifying the approximation error.  For all known estimators, their
actual relationship to the error is spoiled by \emph{oscillation},
i.e., by some additive terms measuring distances between non-discrete
and discrete data.  Remarkably, oscillation may be even greater than the
error.  This flaw directly interferes with the quality assessment and,
on top of that, it weakens results on adaptive methods and complicates
their proofs.  

In this article we introduce  a new approach to a~posteriori error estimation, where oscillation is \emph{error-dominated}, i.e.\ it is bounded by the error of the finite element approximation, up to a multiplicative constant depending on the shape-regularity of the underlying mesh. 

We illustrate this new approach in the simplest case, where the weak solution $u\in H^1_0(\Omega)$ of the Dirichlet-Poisson problem 
\begin{align}
\label{eq:elliptic}
  -\Delta u = \lff \quad\text{in}~\Omega,
    \qquad 
    u=0\quad \text{on}~\partial\Omega
\end{align}
is approximated by the Galerkin approximation $U$ that is continuous and piecewise affine over some simplicial mesh $\grid$ of $\Omega$.  It is instructive to start by recalling the a~posteriori error bounds in terms of the standard residual estimator
\begin{equation}
\label{d:estR}
\estR(U,f,\grid)
\definedas
\left(
\sum_{\elm\in\grid}
h_\elm \norm[L^2(\partial\elm)]{J(U)}^2
+
h_\elm^2 \norm[L^2(\elm)]{f}^2
\right)^{1/2};
\end{equation}
see, e.g., Ainsworth and Oden \cite{AinsworthOden:00} or
Verf\"urth \cite{Verfuerth:2013}.  If $\lff\in L^2(\Omega)$, then the energy norm error $\norm[H^1_0(\Omega)]{u-U}$
and the estimator are almost equivalent.  More precisely, we have
\begin{equation}
\label{estR:almost-equivalence}
 \norm[H^1_0(\Omega)]{u-U}
 \Cleq
 \estR(U,f,\grid),
\qquad
 \estR(U,f,\grid)
 \Cleq
 \norm[H^1_0(\Omega)]{u-U}
 +
 \osc_0(f,\grid),  
\end{equation}
where the interfering oscillation is given by
\begin{equation}
\label{df:std-osc}
 \osc_0(f,\grid)^2
 \definedas
 \sum_{\elm\in\grid} h_\elm^2 \norm[L^2(\elm)]{f-P_{0,\grid}f}^2
\quad\text{with}\quad
 P_{0,\grid}f|_{\elm} \definedas \frac{1}{|\elm|} \int_{\elm} f.
\end{equation}
Let us discuss the relationship of this classical $L^2$-oscillation and the energy
norm error; for the proofs of the nontrival statements, see
\S\ref{sec:oscillation}.  Customarily, oscillation is associated with
higher order.  This idea is supported 
by the following
observation: if $f$ is actually in $H^1(\Omega)$, then 
$\osc_0(f,\grid)=O(h_\grid^2)$ as
$h_\grid\definedas\max_{\elm\in\grid} h_\elm \searrow 0$. 

On any \emph{fixed} mesh however, the oscillation $\osc_0(f,\grid)$
may be arbitrarily greater than the energy norm error
$\norm[H^1_0(\Omega)]{u-U}$.  This is a consequence of the fact that
the $L^2$-norm is strictly stronger than the $H^{-1}$-norm.  The use
of the $L^2$-norm in \eqref{df:std-osc} can be traced back to its use
in the element residual $h_\elm\norm[L^2(\elm)]{f}$ in \eqref{d:estR}
and so it can be motivated by the request for the computability of the
estimator.  In fact, in contrast to an element residual based upon
some local $H^{-1}$-norm of $f$,  this form reduces to the
(approximate) computation of an integral. 

One may think that the use of the $L^2$-norm is the only reason for the
possible relative largeness of oscillations like $\osc_0(f,\grid)$.
Yet, 
Cohen, 
DeVore and 
Nochetto present in \cite{CohenDeVoreNochetto:12} a striking
example which entails that even the $H^{-1}$-oscillation
\begin{equation}
\label{H-1-osc}
  \begin{split}
    \min_{g\in\P_0(\grid)}\norm[H^{-1}(\Omega)]{f-g}^2&
    \\
    \text{with}\quad &\P_0(\grid) \definedas \{g\in L^\infty(\Omega)
    \mid \forall \elm\in\grid\; g|_{\elm} \text{ is constant} \}
  \end{split}
\end{equation}
from Braess~\cite{Braess:2007} and Stevenson~\cite{Stevenson:07} may converge slower
than the error; see Lemma~\ref{L:CDVD} below. Notice that this contradicts the
aforementioned idea that $\osc_0(f,\grid)$ is always of higher order and, moreover,
in view of $\osc_0(f,\grid)\Cleq\estR(U,f,\grid)$, it entails that also the
estimator $\estR(U,f,\grid)$ may decrease sightly slower than the
error.

The key tool to overcome the shortcomings of the above oscillations is a new projection operator $\ProG$ enjoying the following properties; see \S\S\ref{S:towards-err-dom-oscillation}-\ref{S:Construction-P}:
\begin{itemize}
\item $\ProG\lff$ is discrete for any functional $f\in H^{-1}(\Omega)$.  In comparison to $P_{0,\grid}$, the image of $\ProG$ is enriched by the span of the face-supported Dirac distributions and so contains true functionals.
\item $\ProG\lff$ is computable in a local manner. Here computable means that it can be determined 
from the information available in the linear systems for finite element approximations.
\item The local dual norms of the new oscillation $\lff-\ProG\lff$ are  dominated by corresponding local errors.
This property hinges on the face-supported Dirac distributions and on local $H^{-1}$-stability of $\ProG\lff$.
\item In contrast to the local dual norms of the residual $f+\Delta
  U$, the local dual norms of the discretized residual $\ProG{f}+\Delta U$ can be estimated from below and above
  in a computable manner. 
\end{itemize}
Thanks to these properties, we derive in
\S\S\ref{s:info}-\ref{s:apost-bnds} abstract a posteriori bounds such that the oscillation is bounded by
the error.  In \S\ref{sec:comp_sur} we provide several realizations
leading to hierarchical estimators and 
estimators based on local problems or based
on equilibrated fluxes.
Furthermore, in \S\ref{sec:residual} we show that an
extension of the standard residual estimator~\eqref{d:estR} onto the image of $\ProG$ satisfies 
\begin{equation}
\label{err-est-equiv}
 \norm[H^1_0(\Omega)]{u-U}^2
 \eqsim
 \estR(U,\ProG{f},\grid)^2
 +
 \sum_{z\in\vertices} \norm[H^{-1}(\omega_z)]{f-\ProG{f}}^2,
\end{equation}
where $\vertices$ stands for the set of vertices of $\grid$ and
$\omega_z$ is the star around the vertex $z$.  A comparison with
\eqref{estR:almost-equivalence} immediately yields: 
\begin{itemize}	
\item Both $\estR(U,f,\grid)$ and the right-hand side of
  \eqref{err-est-equiv} bound the energy norm error in terms of $U$,
  $f$, and $\grid$.  However, while the latter one is free of
  overestimation, the first one may overestimate, even
  asymptotically. 
\item Since $\ProG{f}$ is discrete and computable in the
  aforementioned sense, we have that $\estR(U,\ProG{f},\grid)$ is also
  computable, while $\estR(U,f,\grid)$ is not. 
\item Equivalence \eqref{err-est-equiv} thus splits the estimation of
  the error in two parts, reflecting the spirit of
  Verf\"urth~\cite[Remark~1.8]{Verfuerth:2013} and
  Ainsworth~\cite[Section~3.1]{Ainsworth:2010}: One part is computable and
  related to the underlying differential operator.  The other one
  depends solely on data; its computation, or rather estimation,
  hinges on a~priori knowledge.  
\end{itemize}

\section{Model problem and discretization} 
\label{S:MpbDiscComp}
%
%
In order to exemplify our new approach to a~posteriori error
estimation, we consider the homogeneous Dirichlet problem for
Poisson's equation and the energy norm error of the associated linear
finite element solution.  The purpose of this section is to recall the
relevant properties of this boundary value problem and
discretization. 

%
%
\medskip We shall use the following 
notation associated with a (Lebesgue) measurable set $\omega$ of
$\R^d$, $d\in\N$.  Given $m\in\N$, we let $L^2(\omega;\R^m)$ denote
the Lebesgue 
space of square integrable functions over $\omega$ with values in
$\R^m$.  We write $\scp[\omega]{v}{w}$ and $\norm[\omega]{\cdot}^2$
for its scalar product and its induced norm.  For $m=1$, we abbreviate
$L^2(\omega;\R)$ to $L^2(\omega)$. 

If $\omega\subset\R^d$ is non-empty and open, $H^1(\omega)$ stands for
the Sobolev space of all functions in $L^2(\omega)$ whose
distributional gradient is also in $L^2(\omega;\R^d)$.  Moreover, we
let $H^1_0(\omega)$ be the closure in $H^1(\omega)$ of all infinitely
differentiable function with compact support in $\omega$.  If the
boundary $\partial\omega$  of $\omega$ is sufficiently regular (e.g.,
Lipschitz), this are all functions in $H^1(\omega)$ with vanishing
trace on the boundary $\partial\omega$.  Thanks to Friedrichs'
inequality,  $H^1_0(\omega)$ is a Hilbert space with scalar product
$\scp[\omega]{\nabla \cdot}{\nabla \cdot}$ and norm
$\norm[\omega]{\nabla \cdot}$.  As usual, $H^{-1}(\omega)$ indicates
the dual space of $H^1_0(\omega)$, i.e.\ the space of linear and
continuous functionals on $H^1_0(\omega)$.  We identify $L^2(\omega)$
with its dual space and thus have 
\begin{equation}
\label{Hilbert-triple}
 H^1_0(\omega)
 \subset
 L^2(\omega)
 \subset
 H^{-1}(\omega).
\end{equation}
The norm of $H^{-1}(\omega)$ is given by 
\begin{equation*}
 \norm[H^{-1}(\omega)]{\lfg}
 :=
 \sup_{w\in H^1_0(\omega)}
  \frac{\dual{\lfg}{w}_\omega}{\norm[\omega]{\nabla w}},
\qquad
 \lfg\in H^{-1}(\omega),
\end{equation*}
where the dual brackets $\dual{\lfg}{w}_\omega := \lfg(w)$, $w\in
H^1_0(\omega)$, extend-restrict the scalar product in $L^2(\omega)$.
If $D\subset\R^d$ is a set such that $\mathring{D}$ is suitable for
one of the preceding notations, we also use $D$ instead of the more
cumbersome $\mathring{D}$, e.g.\ we write also $H^1(D)$ instead of
$H^1(\mathring{D})$. 

\medskip
Let $\Omega$ be an open, bounded and connected subset of $\R^d$ whose
closure can be subdivided into simplices.  We shall omit $\Omega$ in
the notation of dual pairings and norms.  The weak formulation of
\eqref{eq:elliptic} reads as follows: 
\begin{equation}
\label{weak_form}
\begin{aligned}
 \text{Given $\lff\in H^{-1}(\Omega)$, find }
 &u=u_f\in H^1_0(\Omega) \text{ such that }
\\ 
 &\forall v\in H^1_0(\Omega)
\quad
 \scp[]{\nabla u}{\nabla v}
 =
 \dual{\lff}{v}.
\end{aligned}
\end{equation}
In other words: we are looking for the Riesz representation of $\lff$
in $H^1_0(\Omega)$.
Notice that the Riesz representation theorem establishes an isomorphism between the space $H^1_0(\Omega)$ of solutions and the space $H^{-1}(\Omega)$ of loads. In particular, a unique solution
exists not only for $\lff\in L^2(\Omega)$ but for all $\lff\in
H^{-1}(\Omega)$.  This fact suggests that, at least conceptually, an
approximation method for \eqref{weak_form}, along with its
a~posteriori analysis, should cover also loads in $H^{-1}(\Omega)$. 

%
%
\medskip
In order to approximate the solution of \eqref{weak_form}, we use a
Galerkin approximation based upon finite elements.  For the sake of
simplicity, we restrict ourselves to simplicial meshes and lowest
order. 

Let $\grid$ be a simplicial, face-to-face (conforming) mesh of the
domain $\Omega$.  Given an element $\elm\in\grid$, we denote 
by $h_\elm := \diam \elm := \sup_{x,y\in\elm} |x-y|$ its diameter and 
by $\rho_\elm := \sup \{ \diam B \mid B \text{ ball in }\elm \}$ the
maximal diameter of inscribed balls.  In what follows, `$\Cleq$'
stands for `$\leq C$', where the generic constant $C$ may depend on
$d$ and the \emph{shape coefficient} 
\begin{equation*}
 \sigma(\grid)
 :=
 \max_{\elm\in\grid} \sigma_\elm
\quad\text{with}\quad
 \sigma_\elm
 \definedas
 \frac{h_\elm}{\rho_\elm}.
\end{equation*}
In the case of both inequalities `$\Cleq$' and `$\gtrsim$', we shall use `$\simeq$' as shorthand.
 
An interelement face of $\grid$ is a simplex $\side$ with $d$ vertices
arising as the intersection $\side=\elm_1\cap\elm_2$ of two uniquely
determined elements $\elm_1,\elm_2\in\grid$.  Its associated patch is 
\begin{equation}
\label{side_patch}
 \omega_\side \definedas \elm_1\cup\elm_2.
\end{equation}
We let  $\sides =\sides(\grid)$ denote the set of all
$(d-1)$-dimensional interelement faces of $\grid$.  Given
$\side\in\sides$ and $\elm\in\grid$ with $\side\subset\elm$, we write 
\begin{equation}
\label{h-elm-side}
 h_{\elm;\side}
 = 
 \frac{d|\elm|}{|\side|}
 \in
 [\rho_\elm,h_\elm]
\end{equation}
for the height of $\elm$ over $\side$.

Furthermore, $\vertices=\vertices(\grid)$ stands for the set of all
vertices of $\grid$. To any vertex $z\in\vertices$, we associate the
sets 
\begin{equation*}
\label{patches}
 \omega_z
 :=
 \bigcup \{ \elm\in\grid : \elm\ni z\},
\quad
 \sigma_z
 :=
 \bigcup \{\side\in\sides : \side\ni z \},
\end{equation*}
for which we have
\begin{equation}
\label{card-z}
 \#\{\elm\in\grid \mid \elm \ni z\}
 \Cleq
 \#\{\side \in \sides \mid \side \ni z \}
 \Cleq
 1.
\end{equation}
If $\elm\in\grid$ with $\elm \subset\omega_z$ for some $z\in\vertices$, then the diameter $h_z$ of $\omega_z$ verifies
\begin{align}
\label{eq:h}
 h_\elm \le h_z \Cleq h_\elm.
\end{align}
Moreover, if $e$ is a direction, i.e.\ $e\in\R^d$ with $|e|=1$, we
write $h_{z;e}$ for the maximal length of a line segment in
$\omega_z$ with direction $e$. Then 
\begin{equation}
\label{rho-z}
 \tilde{\rho}_z
 \definedas
 \inf_{|e|=1} h_{z;e}
\end{equation}
verifies
\begin{equation}
\label{rho-elm-z}
 \rho_\elm \leq \tilde{\rho}_z \Cleq \rho_\elm 
\end{equation}
whenever $\elm\in\grid$ with $\elm \subset \omega_z$.

%

Let $\P_k$ be the space of polynomials of degree at most
$k\in\N$ over $\R^d$ and let
\begin{align*}
 \P_k(\grid)
 :=
 \big\{V \in L^\infty(\Omega) \mid
  V|_{\elm}\in
  \P_k(\elm) \text{ for all }\elm\in\grid\big\}
\end{align*}
be its piecewise counterpart over $\grid$. The space of continuous, piecewise affine functions over $\grid$ is then
\begin{align*}
 \VG
 := 
 \P_1(\grid)\cap H^1(\Omega)
 =
 \P_1(\grid)\cap C^0(\Bar{\Omega}).
\end{align*}
Its \emph{nodal basis} $\{\phi_z\}_{z\in\vertices}$ is defined by 
\begin{align*}
 \phi_z\in\VG
\quad\text{such that}\quad
 \phi_z(y):=\delta_{zy}~\text{for all}~z,y\in\vertices.
\end{align*}
This basis provides the nodal value representation
\begin{equation*}
 V = \sum_{z\in\vertices} V(z) \phi_z
\end{equation*}
for any $V\in\VG$ and the partition of unity 
\begin{equation}
\label{po1}
 \sum_{z\in\vertices}\phi_z = 1
\quad
 \text{in }\Omega,
\end{equation}
where, for each vertex $z\in\vertices$, we have $\supp\phi_z =
\omega_z$, with skeleton $\sigma_z$.  Finally, we recall that, for any
element $\elm\in\grid$ and any powers $\alpha_z\in\N_0$,
$z\in\vertices\cap\elm$, we have 
\begin{equation}
\label{Integrals_with_Phi_z}
 \int_{\elm} \prod_{z\in\vertices\cap\elm} \phi_z^{\alpha_z}
 =
 \frac{d!\prod_{z\in\vertices\cap\elm} \alpha_z!}{(\sum_{z\in\vertices\cap\elm} \alpha_z + d)!} |\elm|.
\end{equation}

The finite element functions satisfying the boundary condition in \eqref{weak_form} form the space
\begin{align*}
 \VoG
 :=
 \{ V\in\VG \mid
  V(z) = 0 \text{ for all } z\in\vertices\cap\partial\Omega \}
 =
 \P_1(\grid) \cap H^1_0(\Omega).
\end{align*}
The associated \emph{Galerkin approximation} $U=U_{\lff;\grid}$ is characterized by
\begin{equation}
\label{eq:discrete}
 U\in\VoG 
\quad\text{such that}\quad
 \forall V\in\VoG \quad
  \scp[]{\nabla U}{\nabla V}=\dual{\lff}{V}.
\end{equation}
Notice that the right-hand side and so $U$ are well-defined, also for
$\lff\in H^{-1}(\Omega)$, thanks to the conformity of $\VoG$.  C\'ea's
lemma states that the Galerkin approximation is the best approximation
with respect to the energy norm error, i.e.,
\begin{align}
\label{best_approx}
 \norm{\nabla u-\nabla U}
 \le
 \norm{\nabla u -\nabla V}
\qquad\text{for all}~V\in\VoG.
\end{align}

In order to determine the Galerkin approximation $U$, one usually obtains its values at the interior vertices $\vertices_0 := \vertices \cap \Omega$ by solving the symmetric positive definite linear system
\begin{equation*}
 M\alpha = F,
\end{equation*}
where
\begin{equation}
\label{assemble}
 \alpha=(U(z))_{z\in\vertices_0},
\quad
 M=\big(
 \scp[]{\nabla\phi_z}{\nabla\phi_y}
 \big)_{y,z\in\vertices_0},
\quad
 F=(\dual{\lff}{\phi_y})_{y\in\vertices_0}.
\end{equation}
We thus see that the Galerkin approximation $U$ is computable whenever the load evaluations
\begin{equation}
\label{info-for-Galerkin}
 \dual{\lff}{\phi_y}, y \in \vertices_0, \text{ are known exactly.}
\end{equation}

Strictly speaking, these evaluations are in general not computable.  In fact, even if
$\lff\in L^2(\Omega)$ is a function, the evaluation of $\dual{\lff}{\phi_y} =
\int_{\Omega} \lff\phi_y$ requires the computation of an integral,
which in general can be done only approximately by means of numerical
integration.  Notwithstanding, error analyses of approximations like
\eqref{eq:discrete} have proved very useful for the theoretical
understanding and underpinning of finite element methods and are
therefore very common. Accordingly, we shall suppose that the evaluations \eqref{info-for-Galerkin} are known to us. In \S\ref{s:info} below, we will discuss which kind of additional information is used in our a posteriori analysis.

\section{A posteriori analysis with error-dominated oscillation}
\label{S:APost}
We present our new approach to a~posteriori error analysis by deriving
bounds for the energy norm error of the Galerkin approximation
\eqref{eq:discrete}.  The key feature of these bounds is that all
involved terms are dominated by the error. 

\subsection{Residual norms}
Given some load $\lff\in H^{-1}(\Omega)$ and a Galerkin approximation
$U_{\lff;\grid}$, we want to quantify the energy norm error
$\norm[]{\nabla(u_\lff-U_{\lff;\grid})}$, where the exact solution
$u_\lff$ of \eqref{weak_form} is typically unknown to us. 

Our starting point is the so-called \emph{residual} $\Res(\lff;\grid)\in H^{-1}(\Omega)$ given by
\begin{equation*}
 \dual{\Res(\lff;\grid)}{v}
 :=
 \dual{\lff}{v} - \scp[]{\nabla U_{\lff;\grid}}{\nabla v}
\quad\text{for all } v\in H^1_0(\Omega).
\end{equation*}
It is defined in terms of data and the computable Galerkin
approximation and vanishes if and only if the latter equals the exact
solution.  The following lemma shows that appropriately measuring the
size of the residual relates to the error.   
\begin{lem}[Error, residual and load]
\label{L:err_and_res}
We have
\begin{equation*}
 \norm[]{\nabla(u_\lff-U_{\lff;\grid})}
 =
 \norm[H^{-1}(\Omega)]{\Res(\lff;\grid)}
 \leq
 \norm[H^{-1}(\Omega)]{\lff}.
\end{equation*}
\end{lem}
\begin{proof}
Thanks to the differential equation in \eqref{weak_form}, we have, for all $v\in H^1_0(\Omega)$,
\begin{equation}
\label{res-err}
 \dual{\Res(\lff;\grid)}{v}
 =
 \scp[]{\nabla(u_\lff-U_{\lff;\grid})}{\nabla v}
 =
 \dual{-\Delta(u_\lff-U_{\lff;\grid})}{v},
\end{equation}
where $-\Delta$ indicates the distributional Laplacian.  Consequently,
the claimed equality follows from the fact that $-\Delta:H^1_0(\Omega)
\to H^{-1}(\Omega)$ is an isometry (which follows from the
Cauchy-Schwarz inequality in $L^2(\Omega)$ and from testing with
$v=u_\lff-U_{\lff;\grid}$).  The claimed inequality follows by
invoking also \eqref{best_approx}: 
\begin{equation*}
 \norm[]{\nabla(u_\lff-U_{\lff;\grid})}
 \leq
 \norm[]{\nabla u_\lff}
 =
 \norm[H^{-1}(\Omega)]{\lff}. \qedhere
\end{equation*}
\end{proof}

Thus, we aim now at quantifying the dual norm
$\norm[H^{-1}(\Omega)]{\Res(\lff;\grid)}$.  The following simple
observation shows that this task requires much more information than
computing the Galerkin approximation. 

\begin{lem}[Bounding residual norms]
\label{L:info_for_res_norm}
Without any a~priori information on the load $\lff\in H^{-1}(\Omega)$,
the residual norm $\norm[H^{-1}(\Omega)]{\Res(\lff;\grid)}$ cannot be
bounded in terms of a finite number of adaptive evaluations of the
form: $\dual{\lff}{v}$ with $v\in H^1_0(\Omega)$. 
\end{lem}
\begin{proof}
Suppose that the claim is false.  Then, for each $\lff\in
H^{-1}(\Omega)$, there is a number $B(\lff) \geq
\norm[H^{-1}(\Omega)]{\Res(\lff;\grid)}$ which is given in terms of
evaluations $\dual{\lff}{v_i}$, $i=1,\dots,n_\lff$, where the choice
of $v_i$ may depend deterministically on the previous evalutations $\dual{\lff}{v_1}, \dots,
\dual{\lff}{v_{i-1}}$.  Fix some functional $0\neq\lfg\in
H^{-1}(\Omega)$.  Since $H^1_0(\Omega)$ is infinite-dimensional, we
can choose a normalized $w \in H^1_0(\Omega)$ that is
perpendicular to $\VoG$ 
and all test functions $v_i$, $i=1,\dots,n_{\lfg}$ associated with
$\lfg$.  Set $\delta:= 3B(\lfg) (-\Delta)w$ and observe that
$U_{\delta;\grid}=0$ and $\dual{\delta}{v_i}=0$ for all
$i=1,\dots,n_{\lfg}$.  Therefore
$\dual{\ell+\delta}{v_i}=\dual{\ell}{v_i}$ and we obtain the contradiction 
\begin{align*}
 B(\lfg)
 &=
 B(\lfg + \delta)
 \geq
 \norm[H^{-1}(\Omega)]{\Res(\lfg+ \delta;\grid)}
 =
 \norm[H^{-1}(\Omega)]{\delta + \Res(\lfg;\grid)}
\\
 &\geq
 \norm[H^{-1}(\Omega)]{\delta}
  - \norm[H^{-1}(\Omega)]{\Res(\lfg;\grid)}
 \geq
 3 B(\lfg) - B(\lfg)
 =
 2 B(\lfg) > 0. \qedhere
\end{align*}
\end{proof}
\begin{rem}[Load evaluations vs exact integrals]
\label{R:load_evals_vs_exact_int}
A similar yet simpler argument shows that, without any a~priori
information on $f \in  L^2(\Omega)$, also $\norm[]{f}$ cannot be
bounded in terms of adaptive evaluations $\int_{\Omega} fv$ with $v\in
L^2(\Omega)$.  
\end{rem}

Before discussing in
\S\ref{S:towards-err-dom-oscillation} repercussions of Lemma \ref{L:info_for_res_norm}, it is useful to take into
account a further requirement for a~posteriori bounds.  

\subsection{Localized residual norm}
\label{S:loc_res_norm}
Adaptive mesh refinement is an important application of a~posteriori
bounds. It is usually based upon the comparison of local quantities.
Therefore, it is of interest to split a~posteriori bounds, or the residual
norm itself, into local contributions.  

Such a localization appears implicitly, e.g., in the a posteriori error
analysis of Babu\v{s}ka and Miller~\cite{BabuskaMiller:87}. It is based upon
the $W^{1,\infty}$-partition of unity
\eqref{po1} and the orthogonality property: 
\begin{align*}
 \dual{\Res(f;\grid)}{\phi_z}
 =
 0
\qquad\text{for all $z \in \vertices_0 = \vertices \cap \Omega$}.
\end{align*}
We thus introduce the subclass
\begin{equation*}
\mathcal{R}_{\grid}
\definedas
\{ \lfg \in H^{-1}(\Omega) \mid
\forall V\in\VoG\ \dual{\lfg}{V} = 0 \}
\end{equation*}
of residuals associated with Galerkin approximations.  Recall that $\supp\phi_z = \omega_z$ and that $H^{-1}(\omega_z)$ is a shorthand for $H^{-1}(\mathring{\omega}_z)$.
\begin{lem}[Localization]
	\label{L:localization}
	Let $\lfg \in H^{-1}(\Omega)$ be any functional.
	\begin{enumerate}[label=(\roman*)]
		\item\label{L:loci} If $\lfg\in\mathcal{R}_\grid$, then
		\begin{equation*}
		\norm[H^{-1}(\Omega)]{\lfg}^2
		\Cleq 
		\sum_{z\in\vertices}  \norm[H^{-1}(\omega_z)]{\lfg}^2,
		\end{equation*}
		where the hidden constant 
		depends only on $d$ and the shape coefficient $\sigma(\grid)$. 
		\item\label{L:locii} We have
		\begin{equation*}
		\sum_{z\in\vertices} \norm[H^{-1}(\omega_z)]{\lfg}^2
		\leq
		(d+1) \norm[H^{-1}(\Omega)]{\lfg}^2.
		\end{equation*}
	\end{enumerate}
\end{lem}
\begin{proof}
See also Cohen, DeVore, and Nochetto~\cite[\S3.2 and \S3.4]{CohenDeVoreNochetto:12}, Ern and Guermond~\cite[Proposition~31.7]{ErnGuermond:2019} or Blechta, M\'{a}lek, and Vohral\'{i}k~\cite[Theorem~3.7]{BlechtaMalekVohralik:ta}. For the sake of completeness, we provide details. In order to show \ref{L:loci}, we fix an arbitrary $v \in H^1_0(\Omega)$.  In view of the partition of unity \eqref{po1} and $\lfg\in\mathcal{R}_\grid$, we can write
\begin{align}
	\label{representation}
	\dual{\lfg}{v}
	=
	\sum_{z\in\vertices}
	\dual{\lfg}{v\phi_z}
	=
	\sum_{z\in\vertices}
	\dual{\lfg}{(v-c_z)\phi_z},
\end{align}
where the reals $c_z\in\R$ are given by 
\begin{align*}
	c_z
	\definedas
	\frac{\int_{\omega_z}v\phi_z\dx}
	{\int_{\omega_z}\phi_z\dx}
	\text{ for }z\in\vertices_0,
	\quad\text{and}\quad
	c_z=0 \text{ for }z\in\vertices\setminus\vertices_0.
\end{align*}
Thanks to $0\leq\phi_z\leq1$, the inverse estimate
$\norm[L^\infty(\omega_z)]{\nabla\phi_z} \leq \max_{\elm \subset
  \omega_z} \rho_{\elm}^{-1} \Cleq h_z^{-1}$ and the Poincar\'e-Friedrichs
inequality $\norm[\omega_z]{v-c_z} \Cleq h_z \norm[\omega_z]{\nabla
  v}$ (see, e.g., Nochetto and Veeser~\cite[Lemma~4]{NochettoVeeser:12}), we have, for
any $z\in\vertices$, 
\begin{align}
\label{eq:vphiz}	\norm[\omega_z]{\nabla \big( (v-c_z)\phi_z \big)}
	\leq
	\norm[\omega_z]{\nabla v}
	+
	\norm[\omega_z]{v-c_z} \norm[L^\infty(\omega_z)]{\nabla\phi_z}
	\leq
	C_{\sigma(\grid)}
	\norm[\omega_z]{\nabla v},
\end{align}
where the constant $C_{\sigma(\grid)}$ depends only on $\sigma(\grid)$.  Thus, \eqref{representation} leads to 
\begin{equation*}
	|\dual{\lfg}{v}|
	\Cleq
	\sum_{z\in\vertices}
	\norm[H^{-1}(\omega_z)]{\lfg} \norm[\omega_z]{\nabla v}
	\leq\sqrt{d+1}
	\left(
	\sum_{z\in\vertices} \norm[H^{-1}(\omega_z)]{\lfg}^2
	\right)^{1/2}
	\norm[]{\nabla v}
\end{equation*}
and the proof of~\ref{L:loci} is finished.
	
To prove~\ref{L:locii}, we let $v_z \in H^1_0(\omega_z)$ with
$\norm[\omega_z]{\nabla v_z}\leq 1$ for any node $z \in \vertices$ and
set $v=\sum_{z\in\vertices} \dual{\lfg}{v_z}v_z \in H^1_0(\Omega)$. Then 
\begin{equation*}
	\sum_{z\in\vertices}\dual{\lfg}{v_z}^2
	=
	\dual{\lfg}{v}
	\leq
	\norm[H^{-1}(\Omega)]{\lfg} \norm[]{\nabla v},
\end{equation*}
and, with the help of two Cauchy-Schwarz inequalities, 
\begin{align*}
	\norm[]{\nabla v}^2
	&=
	\sum_{\elm\in\grid}
	\sum_{z,y\in\vertices\cap\elm}
	\dual{\lfg}{v_z} \dual{\lfg}{v_y}
	\int_{\elm} \nabla v_z \cdot \nabla v_y
	\\
	&\leq
	\sum_{\elm\in\grid}
	\sum_{z\in\vertices\cap\elm} (d+1) 
	|\dual{\lfg}{v_z}|^2 \norm[\elm]{\nabla v_z}^2
	=
	(d+1) \sum_{z\in\vertices} |\dual{\lfg}{v_z}|^2.
	\end{align*}
	Consequently, we conclude~\ref{L:locii} by taking the suprema over all $v_z$ for all $z\in\vertices$.
\end{proof}

Thus, in the context of adaptive mesh refinement, we are also interested in quantifying the single terms of the \emph{localized residual norm}
\begin{equation}
\label{loc_res_norm}
 \norm[H^{-1}(\grid)]{\Res(\lff;\grid)}^2
 \definedas
 \sum_{z\in\vertices}
  \norm[H^{-1}(\omega_z)]{\Res(\lff;\grid)}^2.
\end{equation}
Of course, we face the same problem for the local residual norms as for the global one.

\begin{cor}[Bounding local residual norms]
\label{C:not-boundable}
Without any a~priori information on $\lff\in H^{-1}(\Omega)$, each local residual norm $\norm[H^{-1}(\omega_z)]{\Res(f,\grid)}$, $z\in\vertices$, cannot be bounded in terms of a finite number of adaptive evaluations of $\lff$.
\end{cor}
\begin{proof}
Replace the domain $\Omega$ by $\omega_z$ in the proof of Lemma
\ref{L:info_for_res_norm} and extend functionals in $H^{-1}(\omega_z)$
by $0$ on the orthogonal complement of $H^1_0(\omega_z)$ in $H^1_0(\Omega)$. 
\end{proof}

\subsection{Towards error-dominated oscillation}
\label{S:towards-err-dom-oscillation}
%
%
In view of Lemma~\ref{L:info_for_res_norm} and
Corollary~\ref{C:not-boundable}, a~posteriori bounds for the residual
norm or its localized variant require knowledge on the load $\lff$ beyond a finite number of evaluations.  The actual knowledge of $\lff$
can be of different nature and, accordingly, may require different
techniques.  Here we want to address only aspects of a~posteriori
error estimation that are independent of the nature of this knowledge.
Correspondingly, we split the residual into an \emph{discretized
  residual} and \emph{data approximation}: 
\begin{equation}
\label{res_splitting}
 \Res(\lff;\grid)
 =
 \big( \ProG\lff + \Delta U_{\lff;\grid} \big)
 +
 \big( \lff - \ProG\lff \big)
\end{equation}
where $\ProG$ maps onto a subspace $\OFDG$ of $H^{-1}(\Omega)$ such that
\begin{itemize}
\item $\norm[H^{-1}(\grid)]{\ProG\lff+\Delta U_{\lff;\grid}}$ can be bounded with the help
of a finite number of evaluations of $\lff$ and
\item the task of bounding $\norm[H^{-1}(\grid)]{\lff - \ProG\lff}$
  hinges only on knowledge of the load~$\lff$; this task may be viewed
  as a matter of approximation theory since, apart from the choice of
  the norm, it is independent of the boundary value problem~\eqref{weak_form}.
\end{itemize}
Here we have used the localized dual norm
$\norm[H^{-1}(\grid)]{\cdot}$ in order to allow for applications in
mesh adaptivity.  It is then desirable that both parts are
dominated by the error,  \ie we have 
\begin{subequations}
\label{sharp_splitting}
\begin{align}
\label{err-dom-aprx-res}
 \norm[H^{-1}(\grid)]{\ProG\lff + \Delta U_{\lff;\grid}}
 &\Cleq
 \norm[]{\nabla(u_\lff-U_{\lff;\grid})}
   ,
\\ \label{err-dom-osc}
 \norm[H^{-1}(\grid)]{\lff-\ProG\lff}
 &\Cleq
 \norm[]{\nabla(u_\lff-U_{\lff;\grid})}
   .
\end{align}
\end{subequations}
In view of Lemma~\ref{L:err_and_res} and Lemma~\ref{L:localization},
the two conditions are equivalent.

The construction of a suitable mapping $\ProG$ is the new twist in our approach.
In order to get first hints on this, let us test out several candidates with necessary
conditions arising from \eqref{err-dom-osc}. 

The proof of Corollary \ref{C:not-boundable} suggests that the problem
lies in the fact that $f$ is taken from an infinite-dimensional space.
The projection $\ProG[0,\grid]$ into discrete data from
\eqref{df:std-osc} is thus a candidate for $\ProG$.  This choice,
however, does not verify \eqref{sharp_splitting}.  In fact, Lemma
\ref{L:err_and_res}, Lemma \ref{L:localization}~\ref{L:locii}, and
\eqref{err-dom-osc} imply the stability estimate 
\begin{equation}
\label{ProG:stability}
 \norm[H^{-1}(\grid)]{\ProG\lff}
 \Cleq
 \norm[H^{-1}(\Omega)]{\lff},
\end{equation}
while $\ProG[0,\grid]\lff$ is not even defined for a general $\lff\in H^{-1}(\Omega)$ (and cannot be continuously extended; cf.\ Lemma~\ref{l:size-cosc}).

This flaw is easily remedied.  For any element $\elm\in\grid$, we replace in \eqref{df:std-osc} the characteristic function $\chi_\elm$ of $\elm$ by the weighted mean
\begin{equation}
\label{psi_elm}
 \psi_\elm
 \definedas
 \frac{(2d+1)!}{d!|\elm|}
 \prod_{z\in\vertices\cap\elm} \phi_z \in H^1_0(\elm)
\quad\text{with}\quad
 \int_{\elm} \psi_\elm = 1
\end{equation}
thanks to \eqref{Integrals_with_Phi_z} and consider
\begin{equation}\label{eq:TProG}
 \TProG[0,\grid]\lff
 \definedas
 \sum_{\elm\in\grid} \dual{\lff}{\psi_\elm}\chi_\elm.
\end{equation}
Since $\psi_\elm\in H^1_0(\elm)\subset H^1_0(\Omega)$ is an admissible test function, the operator $\TProG[0,\grid]$ is defined for all functionals in $H^{-1}(\Omega)$ and satisfies the stability estimate \eqref{ProG:stability}; see Remark \ref{R:stability_TP0} below.

But still, the new operator $\TProG[0,\grid]$ does not verify \eqref{sharp_splitting}.  To see this, consider $\lff=-\Delta V$ with $V\in\VoG$ arbitrary.  We then have
\begin{equation*}
\label{f=DeltaV}
 u_\lff=U_{\lff;\grid}
\end{equation*}
and therefore ${\Res(f;\grid)}=0$ and 
property \eqref{err-dom-osc} entails
\begin{equation}
\label{ProG:IdonVoG}
\forall V\in\VoG
\quad
\ProG(\Delta V) = \Delta V.
\end{equation}
In addition, integration by parts yields that, for all $v\in H^1_0(\Omega)$,
\begin{equation}
\label{DeltaV=}
 \dual{\Delta V}{v}
 =
 -\int_{\Omega} \nabla V\cdot \nabla v
 =
 \sum_{\side\in\sides} \int_{\side} J(V)v \ds,
\end{equation}
where $\ds$ indicates the $(d-1)$-dimensional Hausdorff measure in
$\R^d$ and $J(V)$ is the jump in the normal flux $\nabla V\cdot n$ across
interelement sides.  More precisely, if $\side = \elm_1 \cap \elm_2$
is the intersection of the elements $\elm_1,\elm_2\in\grid$ with
respective outer normals $\normal_1$, $\normal_2$, then $J(V)|_{\side}
\definedas \nabla V|_{\elm_1} \cdot \normal_1 + \nabla V|_{\elm_2}
\cdot \normal_2\in\R$. If $V \neq 0$, then we have also $\Delta V \neq
0$, while \eqref{DeltaV=} yields $\TProG[0,\grid](\Delta V)=0$, in
contradiction with \eqref{ProG:IdonVoG}. Hence \eqref{sharp_splitting}
does not hold for $\TProG[0,\grid]$. 

The two conditions \eqref{ProG:stability} and \eqref{ProG:IdonVoG} are central
to our goals. Although they can be checked without involving the Galerkin approximation \eqref{eq:discrete}, they are also sufficient for \eqref{sharp_splitting}, Incidentally, they
imply that $\ProG$ has to be a near best `interpolation' operator in light of the
Lebesgue lemma.

The failure of \eqref{ProG:IdonVoG} for $\TProG[0,\grid]$ is not
related to the choice of the test functions $\psi_\elm$,
$\elm\in\grid$, but to its range.  In fact, \eqref{DeltaV=} and the
fundamental lemma of calculus of variation show that $\Delta V \not\in
L^2(\Omega)$ whenever $V \neq 0$, while $\TProG[0,\grid](\VoG[\grid])
\subset L^2(\Omega)$.  In other words: to remedy, we have to change
the range. 

Finally, it is desirable that $\ProG$ is a local operator for two
reasons. First, this comes in useful when evaluating
$\ProG$. Second, since $-\Delta$ is a local operator, we have the
following lower bound for the local error: 
\begin{equation}
\label{res-err-loc}
 \norm[H^{-1}(\omega_z)]{\Res(\lff;\grid)}
 \leq
 \norm[\omega_z]{\nabla(u_\lff-U_{\lff;\grid})},
\end{equation}
which follows from testing \eqref{res-err} with all $v$ from $H^1_0(\omega_z)$. This
bound can be exploited 
if we strengthen \eqref{sharp_splitting} to the local conditions
\begin{subequations}
	\label{sharp_splitting_local}
	\begin{align}
	\label{err-dom-aprx-res_local}
	\norm[H^{-1}(\omega_z)]{\ProG\lff + \Delta U_{\lff;\grid}}&\Cleq \norm[H^{-1}(\omega_z)]{\Res(\lff;\grid)}
	,
	\\ \label{err-dom-osc_local}
	\norm[H^{-1}(\omega_z)]{\lff-\ProG\lff}&\Cleq \norm[H^{-1}(\omega_z)]{\Res(\lff;\grid)}
	\end{align}
\end{subequations}
for all $z\in\vertices$. We shall therefore demand the stability
\eqref{ProG:stability} and invariance \eqref{ProG:IdonVoG} in a suitable local manner.

In order to formulate local invariance, let us introduce the following notations associated
with an open subset $\omega\subset\Omega$. 
If $\lfg_1,\lfg_2\in H^{-1}(\Omega)$, we say $\lfg_1=\lfg_2$ on
$\omega$ whenever $\lfg_1(v) = \lfg_2(v)$ for all $v \in
H^1_0(\omega)$.  Moreover, we write $\lfg_1 \in \OFDG$ on $\omega$
when additionally $\lfg_2$ can be chosen such that $\lfg_2\in\OFDG$.  Notice that, thanks to the
fundamental lemma of the calculus of variations, these notions reduce
to the usual ones if $\lfg\in L^2(\Omega)$, i.e. $\lfg(v) =
\int_\Omega gv$ for all $v \in H^1_0(\Omega)$. 

Let us summarize our discussion by a list of desired properties for
the operator $\ProG$ and its range $\OFDG \subset
H^{-1}(\Omega)$, which corresponds to the set of all possible discretized
residuals. This list provides the guidelines for our approach
and choices. Denoting by $\Delta(\VoG) = \{ \Delta V \mid V\in\VoG \}$
the image of $\VoG$ under the distributional Laplacian, we aim for the
following properties: 
\begin{subequations}
\label{desired_properties}
\begin{align}
\label{dp:image-under-Laplacian}
 &\Delta(\VoG) 
 \subset
 \OFDG,
\\ \label{dp:quantification}
 &\text{if $\lfg \in \OFDG$ on $\mathring{\omega}_z$, then
   $\norm[H^{-1}(\omega_z)]{\lfg}$ is quantifiable with a finite number}
\\ \notag
 & \qquad \text{ of evaluations of $\lfg$},
\\ \label{dp:linearity}
 & \ProG \text{ is linear},
\\ \label{dp:computability}
 &\ProG\lff \text{ is locally computable in terms of a finite number of evaluations}
\\ \notag
& \qquad \text{  of $\lff$}, 
\\ \label{dp:local_inv}
 &\text{if $\lfg \in \OFDG$ on $\mathring{\omega}_z$, then $\ProG\lfg = \lfg$ on $\mathring{\omega}_z$,}
\\ \label{dp:local-stab}
 &\norm[H^{-1}(\omega_z)]{\ProG\lfg}
  \Cleq
  \norm[H^{-1}(\omega_z)]{\lfg}~\text{for all}~\lfg\in H^{-1}(\omega_z).
\end{align}
\end{subequations}
Regarding the above discussion, we have that conditions
  \eqref{dp:local-stab},~\eqref{dp:local_inv}
  and~\eqref{dp:image-under-Laplacian} are equivalent
  to~\eqref{sharp_splitting_local}; cf.\ \S\ref{s:apost-bnds}.    
Conditions \eqref{dp:computability} and \eqref{dp:quantification}
allow to quantify the local dual norms of the approximate residual
$\ProG\lff + \Delta U_{\lff;\grid} \in \OFDG$ in a computable manner;
compare also with \S\ref{s:info} below. 

In the next three sections we construct two operators $\ProG$ fulfilling \eqref{desired_properties}. 

\subsection{Discretized residuals and a locally stable biorthogonal system}
\label{S:OFD}
We present a possible choice of the set $\OFDG$ of discretized residuals and introduce an associated biorthogonal system, which is instrumental in constructing a suitable operator $\ProG$ with range $\OFDG$.

We set
\begin{multline}
\label{OFDG}
 \OFDG
 \definedas
 \{ \lfg \in H^{-1}(\Omega) \mid
  \dual{\lfg}{v}
  =
  \sum_{\elm\in\grid}
  \int_\elm c_\elm v\dx +
  \sum_{\side\in\sides}
  \int_\side c_\side v\ds
 \\
  \text{for all } v\in H^1_0(\Omega)
  \text{ with } c_\elm,c_\side\in\R \text{ for }\elm\in\grid, \side\in\sides
 \}.
\end{multline}
Every functional $\lfg\in\OFDG$ is thus constant on each element and on
each face. Obviously, condition \eqref{dp:image-under-Laplacian} is
verified.  More precisely,  $\OFDG$ is a strict superset of
$\Delta(\VoG)$, since in $\Delta(\VoG)$ only certain linear
combinations of the constants $c_\side$, $\side\in\sides$ are allowed.
The fact that these constants are independent in $\OFDG$ facilitates
the definition of $\ProG$. Moreover, we have added the contributions
given by the constants $c_\elm$, $\elm\in\grid$, for comparability
with the classical oscillations and a~posteriori error estimators and because similar
contributions will appear for higher order elements; cf.\ Kreuzer and
Veeser~\cite{Kreuzer.Veeser:ip}.
In spite of these enlargements, we still have $\dim\OFDG < \infty $.
Consequently, an argument as in the proof of Lemma
\ref{L:info_for_res_norm}, which hinges on infinite dimension, is
ruled out.

Let us associate a biorthogonal system with $\OFDG$. To this end, we introduce
the surface Dirac distributions
\begin{subequations}
\label{D(M):basis}
\begin{equation}
 \chi_\side: \left\{\begin{matrix}
  H_0^1(\Omega)
  &\to &\R,
 \\
 v &\mapsto &\int_{\side} v \ds,
 \end{matrix}\right.
\quad
 \side\in\sides,
\end{equation}
and we identify the characteristic functions $\chi_\elm$, $\elm\in\grid$, with their associated distributions
\begin{equation}
 \chi_\elm:\left\{\begin{matrix}
  H_0^1(\Omega)
  &\to &\R,
 \\
  v & \mapsto &\int_{\elm} v \dx,
 \end{matrix}\right.
\quad
 \elm\in\grid.
\end{equation}
\end{subequations}
Notice that the definitions of $\chi_\side$ and $\chi_\elm$ involve
different measures for integration: the $(d-1)$-dimensional Hausdorff
measure for $\chi_\side$ and the $d$-dimensional Lebesgue measure for
$\chi_\elm$.  Correspondingly, each $\chi_\elm$ is absolutely
continuous and each $\chi_\side$ is singular with respect to the
$d$-dimensional Lebesgue measure. 

We collect all elements and interelement faces in the index set
$\mathcal{I}=\mathcal{I}(\grid) \definedas \grid \cup \sides$ and derive in the next
lemma the properties of the functionals $\chi_i$, $i\in\mathcal{I}$,
that are of interest to us. 
\begin{lem}[Basis and scaling]
\label{L:basis-scaling}
The functionals $\chi_i$, $i\in\mathcal{I}$, are a basis of $\OFDG$.
For any element $\elm\in\grid$ and any face $\side\in\sides$
containing a vertex $z \in \vertices$, we have 
\[
 \norm[H^{-1}(\omega_z)]{\chi_\elm}
 \leq
 \vol{\elm}^{1/2} \tilde{\rho}_z
\quad\text{and}\quad
 \norm[H^{-1}(\omega_z)]{\chi_\side}
 \leq
 |\side|^{1/2} \tilde\rho_z^{1/2}
\]
with $\tilde{\rho}_z$ from \eqref{rho-z}.
\end{lem}

\begin{proof}
We will use the Friedrichs inequality
\begin{equation}
\label{Friedrichs}
 \forall v \in H^1_0(\omega_z)
\qquad
 \norm[\omega_z]{v}
 \leq
 \tilde{\rho}_z \norm[\omega_z]{\nabla v}
\end{equation}
and the following trace theorem: if $\side \in \sides$ with $\side \ni z$ and $\normal$ denotes a normal of $\side$, then
\begin{equation}
\label{trace-inequality}
 \forall w \in W^{1,1}_0(\omega_z)
\qquad
 \norm[L^1(\side)]{w}
 \leq
 \frac{1}{2} \norm[L^1(\omega_z)]{
   \nabla w\cdot\normal}.
\end{equation}
Given $\elm \in \grid$ with $\elm \ni z$ and any $v\in H^1_0(\omega_z)$, the Cauchy-Schwarz inequality and \eqref{Friedrichs} yield
\begin{equation*}
 \abs{\dual{\chi_\elm}{v}}
 =
 \abs{\int_{\elm} v \dx}
 \leq
 \vol{\elm}^{1/2} \norm[\omega_z]{v}
 \leq
 \vol{\elm}^{1/2} \tilde{\rho}_z \norm[\omega_z]{\nabla v},
\end{equation*}
which verifies the first claimed inequality.  To show the second one,
fix $\side \in \sides$ with $\side \ni z$ and let again $v \in
H^1_0(\omega_z)$.  Using \eqref{trace-inequality} with $w=v^2$ and
then again \eqref{Friedrichs}, we derive 
\begin{equation*}
 \abs{\dual{\chi_\side}{v}}
 =
 \abs{\int_{\side} v \ds}
 \leq
 \vol{\side}^{1/2} \norm[\side]{v}
 \leq
 |\side|^{1/2} \norm[\omega_z]{v}^{1/2}
  \norm[\omega_z]{
    \nabla v\cdot\normal}^{1/2}
 \leq
 |\side|^{1/2} \tilde\rho_z^{1/2} \norm[\omega_z]{\nabla v}
\end{equation*}
and also the second claimed inequality is proved.
\end{proof}

In order to complete the basis of Lemma \ref{L:basis-scaling} to a biorthogonal system, we use the following test functions: Given any element $\elm\in\grid$, take
\begin{subequations}
\label{biorth-test-fcts}
\begin{equation}\label{psi_element}
 \psi_\elm
 =
 \frac{(2d+1)!}{d!|K|} \prod_{z\in\vertices\cap\elm} \phi_z.
\end{equation}
Given any interelement face $\side\in\sides$, let $z_i$, $i=1,2$,
be the vertices in the patch $\omega_\side$, see~\eqref{side_patch},
that are opposite to
$\side$ and set 
\begin{equation}
\label{psi_side}
 \psi_\side
 \definedas
 \frac{(2d-1)!}{(d-1)!|F|} \left(
  \prod_{z \in \vertices\cap\side} \phi_z
 \right) \left(
 1 - (2d+1) \sum_{i=1}^2 \phi_{z_i}
 \right).
\end{equation}
\end{subequations}

Let us verify that the basis $\chi_i$, $i\in\mathcal{I}$ and the test functions $\psi_i$, $i \in \mathcal{I}$, actually form a biorthogonal system with a crucial stability condition. 
\begin{lem}[Locally stable biorthogonal system]
\label{L:stable-biort}
Together with the basis $\chi_i$, $i\in\mathcal{I}$, the test functions $\psi_i$, $i \in \mathcal{I}$, form a locally stable biorthogonal system:
\begin{enumerate}[label=(\roman*)]
\item\label{L:stable-biort_algebra} We have
\begin{equation*}
 \forall i,j \in \mathcal{I}
 \quad
 \dual{\chi_i}{\psi_j} = \delta_{ij}.
\end{equation*}
\item\label{L:stable-biort_stability} Let $\mathcal{I}_z \definedas \{i\in\mathcal{I} \mid i \ni z \}$ denote the elements and faces containing a vertex $z\in\vertices$. Then
\begin{equation*}
\forall i \in \mathcal{I}_z
\quad
\norm[H^{-1}(\omega_z)]{\chi_i} \norm[\omega_z]{\nabla\psi_i} \leq C_\psi, 
\end{equation*}
where the stability constant $C_\psi$ only depends on $d$ and the shape coefficient $\sigma(\grid)$.
\end{enumerate}
\end{lem}

\begin{proof}
To show~\ref{L:stable-biort_algebra}, we consider the cases of elements
$j\in\grid$ and faces $j\in\sides$ separately. First, let
$\elm\in\grid$ be an element.  As already seen in \eqref{psi_elm}, we
have $\dual{\chi_\elm}{\psi_\elm} = \int_{\elm} \psi_\elm 
= 1$.  Moreover, since $\psi_\elm = 0$ in
$\Omega\setminus\ring{\elm}$, we infer $\dual{\chi_{\elm'}}{\psi_\elm}
= 0$ for any $\elm'\in\grid\setminus\{\elm\}$ and
$\dual{\chi_\side}{\psi_\elm} = 0$ for any $\side\in\sides$. 

Second, fix a face $\side\in\sides$. Using \eqref{Integrals_with_Phi_z}, we obtain
\begin{equation*}
 \dual{\chi_\side}{\psi_\side}
 =
 \frac{(2d-1)!}{(d-1)!|F|} \int_{\side} 
  \prod_{z \in \vertices\cap\side} \phi_z \ds
 =
 1.
\end{equation*}
From $\psi_\side = 0$ in $\Omega \setminus
\ring{\omega}_\side$, where $\omega_\side$ is the patch of the two
elements containing the face $\side$, we infer
$\dual{\chi_{\side'}}{\psi_\side}=0$ for any
$\side'\in\sides\setminus\{\side\}$ and
$\dual{\chi_\elm}{\psi_\side}=0$ for any $\elm\in\grid$ with
$\elm\not\supset\side$. Last, let $\elm\in\grid$ such that
$\elm\supset\side$.  Using again \eqref{Integrals_with_Phi_z}, we
deduce 
\begin{equation*}
 \dual{\chi_\elm}{\psi_\side}
 =
 \frac{(2d-1)!}{(d-1)!|F|}  \left(
  \int_{\elm}  \prod_{z \in \vertices\cap\side} \phi_z \dx
  - (2d+1) \int_{\elm} \prod_{z \in \vertices\cap\elm} \phi_z \dx
 \right)
 =
 0.
\end{equation*}

For \ref{L:stable-biort_stability}, we again treat elements and faces separately. Let $\elm\in\grid$ be an
element containing $z$. The well-known inverse estimate
$\norm[\elm]{\nabla \psi_\elm} \leq C_d \rho_\elm^{-1}
\norm[\elm]{\psi_\elm}$, $\elm \subset \omega_z$ and
\eqref{Integrals_with_Phi_z} imply 
\begin{equation*}
 \norm[\omega_z]{\nabla \psi_\elm}
 =
 \norm[\elm]{\nabla \psi_\elm}
 \leq
 \frac{C_d}{ |\elm|^{1/2} \rho_\elm }.
\end{equation*}
Combining this with the first inequality in Lemma \ref{L:basis-scaling} and \eqref{rho-elm-z}, we obtain the claimed inequality for elements:
\begin{equation*}
 \norm[H^{-1}(\omega_z)]{\chi_\elm} \norm[\omega_z]{\nabla\psi_\elm}
 \leq
 C_d \frac{\tilde{\rho}_z}{\rho_\elm}
 \leq
 C_{d;\sigma(\grid)}.
\end{equation*}

Let $\side\in\sides$ be an interelement face containing $z$ and write
$\side = \elm_1 \cap \elm_2$, where $\elm_1, \elm_2 \in \grid$ are the
two elements containing $\side$. Proceeding as before, we deduce 
\begin{equation}\label{eq:DpsiF}
 \norm[\omega_z]{\nabla\psi_\side}^2
 =
 \sum_{n=1,2} \norm[\elm_n]{\nabla\psi_\side}^2
 \leq
 C_d^{2}
 \sum_{n=1,2} \rho_{\elm_n}^{-2} \norm[\elm_n]{\psi_\side}^2
 \leq
 C_d^{2}
 \sum_{n=1,2} \frac{|\elm_n|}{|\side|^2 \rho_{\elm_n}^2}.
\end{equation}
and
\begin{equation*}
 \norm[H^{-1}(\omega_z)]{\chi_\side} \norm[\omega_z]{\nabla\psi_\side}
 \leq
 C_d \left(
  \sum_{i=1,2} \frac{h_{\elm_n;\side}\tilde{\rho_z}}{\rho_{\elm_n}^2}
 \right)^{1/2}
 \leq
 C_{d;\sigma(\grid)}.
 \qedhere
\end{equation*}
\end{proof}

In what follows, we shall rely only on the properties of the test
functions $\psi_i$, $i\in\mathcal{I}$ expressed in
Lemma~\ref{L:stable-biort}.  In other words: what counts is
not their special form, but the fact that they form a stable
biorthogonal system with the basis $\chi_i$, $i\in\mathcal{I}$, of
$\OFDG$.

\subsection{Construction and properties of $\ProG$}
\label{S:ProG_and_local_dual_norms}
\label{S:Construction-P}
We now propose a possible choice for the projection operator $\ProG$ and verify the desired properties~\eqref{desired_properties}. Set
\begin{equation}
\label{df:ProG}
 \ProG\ell
 =
 \sum_{i\in\mathcal{I}} \dual{\ell}{\psi_i} \chi_i,
\end{equation}
where the functionals $\chi_i$, $i\in\mathcal{I}$, are given by
\eqref{D(M):basis} and the test functions $\psi_i$, $i\in\mathcal{I}$,
by  \eqref{biorth-test-fcts}. Clearly, $\ProG$ is linear and $\ProG\lff$ is locally computable in terms of a finite number of evaluations of $\lff$, \ie we have~\eqref{dp:linearity} and \eqref{dp:computability}. 

The biorthogonality of these functionals and test functions implies the following local counterparts of the algebraic condition \eqref{ProG:IdonVoG}. 
\begin{thm}[Local invariance]
\label{T:local_invariants}
For any functional $\lfg\in H^{-1}(\Omega)$, element $\elm\in\grid$, and side $\side\in\sides$, the operator $\ProG$ does not change the following discrete restrictions:
\begin{enumerate}[label=(\roman*)]
\item \label{T:invariantsi} If $\lfg \in \OFDG$ on $\mathring{\elm}$, then  $\ProG\lfg = \lfg $ on $\mathring{\elm}$.
\item \label{T:invariantsii} If $\lfg \in \OFDG$ on $\mathring{\omega}_\side$, then $\ProG\lfg = \lfg$ on $\mathring{\omega}_\side$. 
\end{enumerate}
\end{thm}
\begin{proof}
Let $\lfg = c\chi_K$ on $\mathring{\elm}$ with $c\in\R$. For any $i\in\mathcal{I}$, we have
$\dual{\lfg}{\psi_i} = c \int_\elm \psi_i = c \delta_{\elm,i}$ by means of
Lemma~\ref{L:stable-biort}~\ref{L:stable-biort_algebra}.  Consequently,
$\ProG\lfg = c \chi_\elm$ on $\mathring{\elm}$, which proves~\ref{T:invariantsi}.

To show~\ref{T:invariantsii}, let $\elm_1,\elm_2\in\grid$ be the two
elements containing $\side$ and let $\lfg = c \chi_\side +
\sum_{i=1,2} c_i \chi_{\elm_i}$ on $\mathring{\omega}_\side$ with $c,
c_1, c_2 \in \R$.  Using again
Lemma~\ref{L:stable-biort}~\ref{L:stable-biort_algebra}, we observe 
\begin{equation*}
 \dual{\lfg}{\psi_\side}
 =
 c \dual{\chi_\side}{\psi_\side}
  + \sum_{i=1,2} c_i \dual{\chi_{K_i}}{ \psi_\side}
 =
 c
\quad\text{and}\quad
 \dual{\lfg}{\psi_{\elm_i}}
 =
 c_i
\quad\text{for }i=1,2
\end{equation*}
and $\dual{\ell}{\psi_i} = 0$ for all $i\in\mathcal{I}\setminus\{F,K_1,K_2\}$. Consequently, 
\begin{equation*}
 \ProG\lfg
 =
 c \chi_\side + \sum_{i=1,2} c_i \chi_{\elm_i}
 =
 \lfg
\quad\text{on }\mathring{\omega}_\side
\end{equation*}
and also~\ref{T:invariantsii} is verified.
\end{proof}

Theorem \ref{T:local_invariants} implies in particular \eqref{dp:local_inv}.  Moreover, it has the following global consequences.
\begin{cor}[Global invariance]
\label{C:invariance}
The operator $\ProG$ is a linear projection onto the discretized residuals $\OFDG$ from \eqref{OFDG}. In particular, 
we have
\begin{align*}
 \ProG (\Delta V) = \Delta V
\qquad\text{and}\qquad
 \ProG (f) = f.
\end{align*}
for any $V\in\VoG$ and any $\grid$-piecewise constant function $f \in
\P_0(\grid)$.
\end{cor}

Next, we verify the local stability~\eqref{dp:local-stab} of $\ProG$. As a side product, we also obtain the local stabilty of the operator $\TProG[0,\grid]$, which was left open in \S\ref{S:towards-err-dom-oscillation}.
\begin{thm}[Local stability]
\label{T:loc_stability}
The linear projection $\ProG$ is locally $H^{-1}$-stable: for any functional $\lfg \in H^{-1}(\Omega)$ and any vertex $z\in\vertices$, we have
\begin{equation*}
 \norm[H^{-1}(\omega_z)]{\ProG\lfg}
 \Cleq
 \norm[H^{-1}(\omega_z)]{\lfg},
\end{equation*}
where the hidden constant depends only on $d$ and $\sigma(\grid)$.
\end{thm}
\begin{proof}
Given $v \in H^1_0(\omega_z)$, we derive
\begin{align*}
 |\dual{\ProG\lfg}{v}|
 &\leq
 \sum_{i \in \mathcal{I}_z} |\dual{\lfg}{\psi_i} \dual{\chi_i}{v}|
 \leq
 \sum_{i \in \mathcal{I}_z}
  \norm[H^{-1}(\omega_z)]{\lfg}\norm[\omega_z]{\nabla\psi_i}
  \norm[H^{-1}(\omega_z)]{\chi_i}\norm[\omega_z]{\nabla v}
\\
 &\Cleq \norm[H^{-1}(\omega_z)]{\lfg} \norm[\omega_z]{\nabla v},
\end{align*}
where we used Lemma \ref{L:stable-biort}~\ref{L:stable-biort_stability} and
$\#\mathcal{I}_z \Cleq 1$. 
\end{proof}

\begin{rem}[Stability of {$\TProG[0,\grid]$}]
\label{R:stability_TP0}
The argument in the proof of Theorem \ref{T:loc_stability} also shows
that $\TProG[0,\grid]$ is locally $H^{-1}$-stable. In fact, one simply
replaces $\ProG$ by $\TProG[0,\grid]$ and the index set
$\mathcal{I}_z$ by $\mathcal{I}_z \cap \grid$.  
\end{rem}

Let us conclude this section with the following further remarks on the linear projection~$\ProG$.

\begin{rem}[Orthogonality]
\label{R:ProG-orthogonality}
For any $\lfg \in H^{-1}(\Omega)$, the functional $\lfg - \ProG\lfg$
is orthogonal to $\text{span}\,\{\psi_i \mid i \in \mathcal{I}\}$.
This a immediate consequence of Lemma
\ref{L:stable-biort}\,\ref{L:stable-biort_algebra}. 
\end{rem}

\begin{rem}[Adjoint of $\ProG$]\label{R:adjoint}
Formally, the adjoint of $\ProG$ is given by
\begin{equation*}
 \ProG^*v = \sum_{i\in\mathcal{I}} \dual{\chi_i}{v} \psi_i,
\qquad
 v \in H^1_0(\Omega).
\end{equation*}
Here Lemma \ref{L:stable-biort}~\ref{L:stable-biort_algebra} implies
\begin{equation}\label{eq:propsProG*}
 \int_{\elm} \ProG^*v = \int_{\elm} v
\quad\text{and}\quad
 \int_{\side} \ProG^*v = \int_{\side} v
\end{equation}
for all elements $\elm\in\grid$ and interelement faces
$\side\in\sides$.  The operator $\ProG^*$ and these conditions, which
characterize it, were used in Veeser~\cite{Veeser:02} to derive an
a~posteriori error upper bound in terms of a hierarchical
estimator. That argument, as well as Morin, Nochetto, and
Siebert~\cite[Theorem~3.6]{MoNoSi:03} and
Verf\"urth~\cite[(3.14)]{Verfuerth:96}, is closely related to
Theorem~\ref{T:Quantifying_local_dual_norms} below. 
\end{rem}

\subsection{Required a~priori information, an alternative to $\ProG$, and quantification of the discretized residual}
\label{s:info}
%
The purpose of this section is twofold. First, we illustrate which
type of a~priori information on $\lff$ in \eqref{weak_form} is needed to
carry out our approach, presenting also a possible
alternative to~$\ProG$.
Second, we show that a stable biorthogonal system is
not only useful to construct $\ProG$, but also to quantify the local
dual norms of discretized residuals. 

Clearly, the operator $\ProG$ of \S\ref{S:Construction-P} can be applied to the right-hand side $\lff$ of \eqref{weak_form} whenever
\begin{equation}
\label{avail_info}
 \dual{\lff}{\psi_i},\, i\in\mathcal{I}, \text{ are known exactly.}
\end{equation}
In order to ensure a meaningful discretized residual, this
information goes beyond \eqref{info-for-Galerkin}, the information necessary for the Galerkin
approximation \eqref{eq:discrete} on the mesh $\grid$; it is available, e.g., when one is
able to compute the counterpart of \eqref{eq:discrete} of order $d+1$
over $\grid$. 

There are other possibilities to obtain a meaningful discretized residual. The following one fits particularly well to \eqref{info-for-Galerkin} in the
context of mesh adaptivity. Suppose that we are given an initial mesh
and a refinement procedure such that the set $\mathbb{M}$ of all
refined meshes form a shape-regular family. Furthermore, suppose that,
for any mesh $\grid \in \mathbb{M}$, there is a refinement $\widetilde\grid \in
\mathbb{M}$ with vertices $\vertices(\widetilde\grid)$ 
that satisfies the following properties:
\begin{subequations}
	\label{mesh-adaptivity}
	\begin{alignat}{2}
	&\forall \widetilde{\elm} \in \widetilde{\grid} \; \exists \elm \in \grid &\quad&\text{with}\quad
	 \widetilde{\elm} \subset \elm \text{ and } h_\elm \Cleq h_{\widetilde{\elm}},
	\\ \label{interior-node}
	&\forall i \in \mathcal{I}(\grid)
   \; \exists \widetilde{z} \in \vertices(\widetilde{\grid})
   &\quad&\text{such that}\quad
	 \widetilde{z} \text{ is interior to } i.
	\end{alignat}
\end{subequations}
Let us now fix a mesh $\grid \in \mathbb{M}$ and a refinement
$\widetilde{\grid}\in\mathbb{M}$ satisfying
\eqref{mesh-adaptivity}. For any $i
\in \mathcal{I}(\grid)$, using \eqref{interior-node}, we fix
a vertex $\widetilde{z} \in \vertices(\widetilde{\grid})$ interior to $i$ and denote by
$\widetilde{\phi}_{\widetilde{z}}$ its associated hat function in
$\VG[\widetilde{\grid}]$. We then obtain counterparts
$\widetilde{\psi}_i$, $i \in \mathcal{I}$, of the test functions
$\psi_i$, $i \in \mathcal{I}$, by using these hat functions with a
suitable scaling in place of the element and faces bubble functions in
\eqref{biorth-test-fcts} such that the following lemma holds. We skip the
technical details, referring to Morin, Nochetto and Siebert~\cite{MoNoSi:02}
and Veeser~\cite{Veeser:02}. 
\begin{lem}[Another locally stable biorthogonal system]
	\label{L:alt-stable-biort}
	Together with the basis $\chi_i$, $i\in\mathcal{I}$, the test functions $\widetilde{\psi}_i$, $i \in \mathcal{I}$, form a locally stable biorthogonal system:
	\begin{enumerate}[label=(\roman*)]
		\item\label{L:alt-stable-biort_algebra} We have
		\begin{equation*}
		\forall i,j \in \mathcal{I}
		\quad
		\dual{\chi_i}{\widetilde{\psi}_j} = \delta_{ij}.
		\end{equation*}
		\item\label{L:alt-stable-biort_stability} Let $\mathcal{I}_z = \{i\in\mathcal{I} \mid i \ni z \}$ denote the elements and faces containing a vertex $z\in\vertices$. Then
		\begin{equation*}
		\forall i \in \mathcal{I}_z
		\quad
		\norm[H^{-1}(\omega_z)]{\chi_i} \norm[\omega_z]{\nabla\widetilde{\psi}_i} \leq C_{\tilde\psi}, 
		\end{equation*}
		where the stability constant $C_{\tilde\psi}$ only depends on $d$ and the shape coefficient $\sigma(\grid)$.
	\end{enumerate}
\end{lem}
Thus, the operator
\begin{equation}
\label{alt-projection}
 \widetilde{\mathcal{P}}_\grid \ell
 \definedas
 \sum_{i \in \mathcal{I}} \dual{\ell}{\widetilde{\psi_i}} \chi_i
\end{equation}
defines an alternative to $\ProG$ and the properties \eqref{desired_properties} without \eqref{dp:quantification} can be established as for $\ProG$. The operator $\widetilde{\mathcal{P}}_\grid$ can be evaluated on any mesh $\grid\in\mathbb{M}$ whenever
\begin{equation}
\label{alt-avail_info}
\forall \widetilde{\grid}\in\mathbb{M} \; \forall z\in\vertices_0(\widetilde{\grid}) \quad
\dual{\lff}{\widetilde{\phi}_z} \text{ are known exactly,}
\end{equation}
where $\{\widetilde{\phi}_z\}_{z\in\vertices_0(\widetilde{\grid})}$ denotes the nodal basis
of $\V_0(\widetilde{\grid})$. This is exactly \eqref{info-for-Galerkin} for all meshes in $\mathbb{M}$. Consequently, it is also needed to ensure that an adaptive algorithm with the above refinement procedure can always compute the Galerkin approximation \eqref{eq:discrete}.

Let us now turn to the quantification of the discretized residual and verify \eqref{dp:quantification}, considering a general locally stable biorthogonal system.
\begin{thm}[Quantifying local dual norms]
\label{T:Quantifying_local_dual_norms}
Let $\psi_i$, $i\in\mathcal{I}$, be the test functions from Lemma~\ref{L:stable-biort} or Lemma~\ref{L:alt-stable-biort}. 
If $\lfg \in \OFDG$ on a star $\omega_z$, then the corresponding local dual norm can be quantified by a finite number of evaluations:
\begin{equation*} 
	\frac{1}{d+1}
	\sum_{i\in\mathcal{I}_z} \left|
	\dual{\lfg}{\frac{\psi_i}{\norm[]
			{\nabla\psi_i}}}
	\right|^2
	\leq
	\norm[H^{-1}(\omega_z)]{\lfg}^2
	\Cleq
	\sum_{i\in\mathcal{I}_z} \left|
	\dual{\lfg}{\frac{\psi_i}{\norm[
			]{\nabla\psi_i}}}
	\right|^2
\end{equation*}
where the hidden constant depends on $d$, $\sigma(\grid)$, and $C_\psi$.
\end{thm}

\begin{proof}
Let us first prove the lower bound, which holds for any arbitrary functional $\lfg \in H^{-1}(\Omega)$.
In fact, the definition of the dual norm readily yields 
\begin{equation}
\label{single_low_bd}
	\left|
	\dual{\lfg}{\frac{\psi_i}{\norm[]
			{\nabla\psi_i}}}
	\right|
	\leq
	\norm[H^{-1}(\supp\psi_i)
	 ]{\lfg} 
\end{equation}
for any $i\in\mathcal{I}_z$. Notice that the essential supremum of
$x \mapsto \#\{i\in\mathcal{I}_z \mid \supp\psi_i \ni x \}$ is bounded by $d+1$. Arguing as in the proof of
Lemma~\ref{L:localization}~\ref{L:loc2ii}, we therefore obtain
\begin{equation}
\label{abs_low_loc}
 \sum_{i \in \mathcal{I}_z} \norm[H^{-1}(\operatorname{supp}\psi_i)]{\lfg}^2
 \leq
 (d+1) \norm[H^{-1}(\omega_z)]{\lfg}^2
\end{equation}
and the proof of the lower bound is finished.
	
To show the upper bound, we (need to) assume that $\ell\in\OFDG$ on $\omega_z$. Given $v \in
H^1_0(\omega_z)$, we can then write
\begin{equation*}
	\dual{\lfg}{v}
	=
	\sum_{i\in\mathcal{I}_z} c_i \dual{\chi_i}{v}
	\quad\text{with}\quad
	c_i \in \R.
\end{equation*}
In light of the biorthogonality, we have $c_i = \dual{\lfg}{\psi_i}$.  Using also the local stability of the biorthogonal system, we infer
\begin{align*}
	|\dual{\lfg}{v}|
	&\leq
	\sum_{i\in\mathcal{I}_z} |\dual{\lfg}{\psi_i} \dual{\chi_i}{v}|
	\\
	&\leq
	\sum_{i\in\mathcal{I}_z}
	\norm[\omega_z]{\nabla\psi_i} \norm[H^{-1}(\omega_z)]{\chi_i} \left|
	\dual{\lfg}{\frac{\psi_i}{\norm[
			]{\nabla\psi_i}}}
	\right|
	\norm[\omega_z]{\nabla v}
	\\
	&\leq
	C_{\psi} \left(
	\sum_{i\in\mathcal{I}_z} \left| \dual{\lfg}{\frac{\psi_i}{\norm[
			]{\nabla\psi_i}}}
	\right|\right) \norm[\omega_z]{\nabla v}.
\end{align*}
Since the solid angle of every simplex
containing $z$ is bounded away from $0$ in terms of $d$ and the shape
coefficient $\sigma(\grid)$, we have $\#\mathcal{I}_z \leq
C_{\sigma(\grid)}$.  Consequently, the Cauchy-Schwarz inequality on the
sum implies the desired upper bound.
\end{proof}

Theorem \ref{T:Quantifying_local_dual_norms} implies the missing \eqref{dp:quantification} for both operators $\ProG$ and $\widetilde{\mathcal{P}}_\grid$ and, in accordance with \S\ref{S:towards-err-dom-oscillation}, we have splittings of the local residual norms with the desired properties. Notice that, in view of the discussion of this section and Corollary \ref{C:not-boundable}, bounding the terms
\begin{equation*}
 \norm[H^{-1}(\omega_z)]{\ProG\lff-\lff}
 \text{ or }
 \norm[H^{-1}(\omega_z)]{\widetilde{\mathcal{P}}_\grid\lff-\lff} 
\end{equation*}
cannot be done in general with a finite number of evaluations of the load $\lff$. Notably, these terms involve only the load and the discretized residuals
\begin{equation*}
 \norm[H^{-1}(\omega_z)]{\ProG\lff+\Delta U_{\lff;\grid}}
 \text{ or }
  \norm[H^{-1}(\omega_z)]{\widetilde{\mathcal{P}}_\grid\lff+\Delta U_{\lff;\grid}}
\end{equation*}
can be quantified with finite information, which, in light of Remark
\ref{R:load_evals_vs_exact_int}, is less than the information required
for evaluating local $L^2$-norms of the load $\lff$. 

\subsection{A~posteriori error bounds}
\label{s:apost-bnds}
We now summarize our preceding results by deriving 
a~posteriori error bounds. The resulting bounds are defined for any load
$f\in H^{-1}(\Omega)$ and the oscillation is dominated  by the error.

The following statements remain correct if $\ProG$ is replaced by
$\widetilde{\mathcal{P}}_\grid$ from \eqref{alt-projection}.
%
\begin{thm}[Abstract upper bound]
\label{T:ubd}
For any functional $\lff \in H^{-1}(\Omega)$ and any conforming mesh $\grid$, we have
\begin{equation*}
 \norm[]{\nabla(u_\lff - U_{\lff;\grid})}^2
 \Cleq
 \sum_{z\in\vertices} 
  \norm[H^{-1}(\omega_z)]{\ProG\lff+\Delta U_{\lff;\grid}}^2
  +
  \norm[H^{-1}(\omega_z)]{\ProG\lff-\lff}^2.
\end{equation*}
Each local dual norm 
$\norm[H^{-1}(\omega_z)]{\ProG\lff-\Delta U_{\lff;\grid}}$ of the
discretized residual can be quantified with a finite
number of evaluations of $\lff$, while the quantification of the
local dual norms \norm[H^{-1}(\omega_z)]{\ProG\lff-\lff} of the oscillation
requires additional a~priori information on $\lff$. 
\end{thm}

\begin{proof}
Lemma \ref{L:err_and_res}, Lemma \ref{L:localization} and a triangle
inequality imply the claimed bound. 
Recalling that 
\begin{equation*}
\ProG\lff + \Delta U_{\lff;\grid} \in \OFDG,
\end{equation*}
Lemma \ref{T:Quantifying_local_dual_norms} and Corollary \ref{C:not-boundable} ensure
the statements about the quantification of the two parts of the bound.
\end{proof}

In contrast to previous results available in literature, the complete upper bound
in Theorem \ref{T:ubd} is also a lower bound, even locally. 

\begin{thm}[Abstract local lower bounds]\label{T:lbd}
For any functional $f\in H^{-1}(\Omega)$ and any conforming mesh
$\grid$, the discretized residual and the oscillation are locally
dominated by the error: for every vertex $z\in\vertices$, we have
\begin{equation*}
 \norm[H^{-1}(\omega_z)]{\ProG\lff + \Delta U_{\lff;\grid}}
 \Cleq
 \norm[\omega_z]{\nabla(u_\lff - U_{\lff;\grid})}
\end{equation*}
and
\begin{equation*}
 \norm[H^{-1}(\omega_z)]{\ProG\lff - \lff}
 \Cleq
 \norm[\omega_z]{\nabla(u_\lff - U_{\lff;\grid})}.
\end{equation*}
\end{thm}

\begin{proof}
In light of \eqref{res-err-loc}, the first claimed inequality follows from the triangle inequality and the second one.  The latter is a consequence of Theorems \ref{T:local_invariants} and \ref{T:loc_stability} and \eqref{res-err-loc}:
\begin{align*}
 \norm[H^{-1}(\omega_z)]{\ProG\lff - \lff}
 &\leq
 \norm[H^{-1}(\omega_z)]{\ProG(\lff + \Delta U_{\lff;\grid})}
 +
 \norm[H^{-1}(\omega_z)]{ \lff - \Delta U_{\lff;\grid} }
\\
 &\Cleq
 \norm[H^{-1}(\omega_z)]{ \lff - \Delta U_{\lff;\grid} }
 \Cleq
 \norm[\omega_z]{\nabla(u_\lff - U_{\lff;\grid})}.
\qedhere
\end{align*}
\end{proof}

Squaring and summing, we readily get the global lower bounds.

\begin{cor}[Abstract global lower bounds]\label{C:lbd}
For any functional $f\in H^{-1}(\Omega)$ and any conforming mesh $\grid$, the
discretized residual and the oscillation are globally dominated by the error in that 
\begin{equation*}
 \sum_{z\in\vertices}
  \norm[H^{-1}(\omega_z)]{\ProG\lff + \Delta U_{\lff;\grid}}^2
 \Cleq
 \norm[]{\nabla(u_\lff - U_{\lff;\grid})}^2
\end{equation*}
and
\begin{equation*}
 \sum_{z\in\vertices}
  \norm[H^{-1}(\omega_z)]{\ProG\lff - \lff}^2
 \Cleq
 \norm[]{\nabla(u_\lff - U_{\lff;\grid})}^2.
\end{equation*}
\end{cor}

To summarize: if we are able to quantify the oscillation terms
$\norm[H^{-1}(\omega_z)]{\ProG\lff - \lff}$, $z\in\vertices$, then the
right-hand side in Theorem \ref{T:ubd} is a truly equivalent
a~posteriori error estimator. 

\begin{rem}[Surrogate oscillation]
\label{R:surrogate-osc}
The quantification of the local dual norms
$\norm[H^{-1}(\omega_z)]{\ProG\lff - \lff}$, $z\in\vertices$, of the oscillation
appears to be a difficult matter. In \cite[Section~7]{CohenDeVoreNochetto:12}, Cohen,
DeVore, and Nochetto consider similar terms for special $\lff$ and resort to surrogates
that can be approximated with the help of numerical integration. Those surrogates
hinge on additional regularity of $\lff$, which entails the risk of
overestimation; cf.\ Lemma \ref{l:size-cosc} below.
\end{rem}

\subsection{Classical versus error-dominated oscillation}
\label{sec:oscillation}
%
In this section we compare the error-dominated oscillation $(\sum_{z\in\vertices} \norm[H^{-1}(\omega_z)]{\ProG\lff - \lff}^2)^{1/2}$ with the classical $L^2$- and $H^{-1}$-oscillation,
\begin{equation*}
 \osc_0(\lff,\grid)
\quad\text{and}\quad
 \min_{g \in\P_0(\grid)}  \norm[H^{-1}(\Omega)]{\lff - g},
\end{equation*}
from \eqref{df:std-osc} and \eqref{H-1-osc} in the introduction. Doing so,
we verify statements of the introduction and substantiate the advantages of
the stability and invariance properties of the operator $\ProG$.

Let us first show that the error-dominated oscillation is always smaller, up to a
multiplicative constant, than both classical oscillations. To this end, 
let $\lff\in H^{-1}(\Omega)$ and let  $g\in\P_0(\grid)$
be an arbitrary piecewise constant 
approximation over $\grid$. The local invariance and stability properties of
$\ProG$ in Theorems~\ref{T:local_invariants} and \ref{T:loc_stability} imply
that, for all $z\in\vertices$, 
\begin{equation}
\label{err-dom<classical;common}
\begin{aligned}
  \norm[H^{-1}(\omega_z)]{\lff -\ProG\lff}
  &\le
  \norm[H^{-1}(\omega_z)]{\lff
    -g}+\norm[H^{-1}(\omega_z)]{\ProG(g-\lff)}
    \\
  &\Cleq  \norm[H^{-1}(\omega_z)]{\lff
    -g}.
\end{aligned}
\end{equation}
Combining this with Lemma~\ref{L:localization}~\ref{L:locii} and minimizing over $g$, we obtain the bound in terms of the classical $H^{-1}$-oscillation:
\begin{subequations}\label{eq:osc<osc0}
  \begin{align}\label{eq:osc<osc0a}
    \sum_{z\in\vertices}\norm[H^{-1}(\omega_z)]{\lff
    -\ProG\lff}^2\Cleq
 	\min_{g \in\P_0(\grid)} \norm[H^{-1}(\Omega)]{\lff -g}^2.
\end{align}
To show the other bound, suppose $f\in L^2(\Omega)$. Making use of the orthogonality of $P_{0,\grid}$ and \Poincare inequalities
in the elements of $\omega_z$, we deduce
\begin{equation*}
 \norm[H^{-1}(\omega_z)]{\lff - P_{0,\grid}\lff}^2
 \Cleq
 \sum_{\elm\subset\omega_z} h_\elm^2\norm[\elm]{\lff -P_{0,\grid}\lff }^2,
\end{equation*}
which together with \eqref{err-dom<classical;common} gives the bound in
terms of the $L^2$-oscillation:
\begin{align}\label{eq:osc<osc0b}
 \sum_{z\in\vertices}\norm[H^{-1}(\omega_z)]{\lff -\ProG\lff}^2
 \Cleq
 \sum_{\elm\in\grid}h_\elm^2\norm[\elm]{\lff -P_{0,\grid}\lff }^2=\osc_0(\lff,\grid)^2.
\end{align}
\end{subequations}

The converse bounds of \eqref{eq:osc<osc0} do not hold. For the classical
$L^2$-oscillation, this applies even on a fixed mesh and is in particular
due to stability issues. The following lemma provides an illustration, relating directly to the error instead of the error-dominated oscillation.
\begin{lem}[Overestimation of classical $L^2$-oscillation]
\label{l:size-cosc}
For any conforming mesh $\grid$, there exists a sequence
$(f_k)_{k}\subset L^2(\Omega)$ such that
\begin{align*}
 \frac{\osc_0(\lff_k,\grid)}{ \norm[]{\nabla(u_{\lff_k} - U_{{\lff_k};\grid})}}
 \to
 \infty\qquad\text{as}~k\to\infty.
\end{align*}
\end{lem}

\begin{proof}
Choose $\lff\in H^{-1}(\Omega)\setminus L^2(\Omega)$.  Since
$L^2(\Omega)$ is dense in $H^{-1}(\Omega)$, there exists a
sequence $(f_k)_k\subset L^2(\Omega)$ such that $f_k\to f$ in
$H^{-1}(\Omega)$.  On the one hand, the energy norm errors
$\norm[]{\nabla(u_{\lff_k} - U_{{\lff_k};\grid})}$ 
are uniformly bounded with respect to $k$. 
On the other hand, in view of $\lim_{k\to\infty} \norm[L^2(\Omega)]{f_k}=\infty$,
the oscillation $\osc_0(f_k,\grid)$ becomes arbitrarily large for
$k\to\infty$.
\end{proof}

In the case of the classical $H^{-1}$-oscillation, \eqref{eq:osc<osc0a} cannot be inverted because of invariance issues. Let us illustrate this again by the relationship to the Galerkin error. Consider
\begin{equation}
\label{loads-for-overest}
 \lff=-\Delta V \text{ for some } V\in\VoG[\grid^\dagger]\setminus\{0\},
\end{equation}
where $\grid^\dagger$ is some conforming simplicial mesh of $\Omega$.
For any conforming refinement $\grid$ of $\grid^\dagger$,  we then have $u_\lff = V = U_{\lff;\grid}$ and $\lff \not\in \P_0(\grid)$. Hence
\begin{equation*}
 \norm[]{\nabla(u_\lff - U_{\lff;\grid})} = 0
 <
 \min_{g\in\P_0(\grid)} \norm[H^{-1}(\Omega)]{\lff - g},
\end{equation*}
where the classical $H^{-1}$-oscillation can be made arbitrarily large for a given $\grid$ but decreases to $0$ under suitable refinement. One could argue that the (neighborhoods of the) loads \eqref{loads-for-overest} are very special, in particular because the optimal convergence rate of \eqref{loads-for-overest} is formally $\infty$. Here is another example based upon Cohen, DeVore, and Nochetto~\cite[Section~6.4]{CohenDeVoreNochetto:12}, where the optimal nonlinear convergence rate for the error is finite and often encountered in practice.

\begin{lem}[Another overestimation of classical $H^{-1}$-oscillation]
\label{L:CDVD}
Let $\Omega=(0,1)^2$. There is a functional $\lff\in H^{-1}(\Omega)$ and a sequence $(L_n)_n$ with $\log n \gtrsim L_n \to \infty$ as $n\to \infty$ such that
\begin{subequations}
\label{eq:CohenDeVoreNochetto}
\begin{equation}
\label{eq:CohenDeVoreNochetto-err}
 \min_{\#\grid\le n}\norm[]{\nabla(u_\lff - U_{{\lff};\grid})}
 \Cleq
 n^{-1/2},
\end{equation}
and
\begin{equation}
\label{eq:CohenDeVoreNochetto-osc}
 \min_{\#\grid\le n} \min_{g \in\P_0(\grid)} \Bigg(  \sum_{z\in\vertices(\grid)}\norm[H^{-1}(\omega_z)]{\lff-g}^2
 \Bigg)^{1/2}
 \ge
 L_n\, n^{-1/2}, 
\end{equation}
\end{subequations}
where $\grid$ varies in all meshes created by recursive or iterative newest
vertex bisection of some conforming initial mesh $\grid_0$ of
$\Omega$. 
\end{lem}
\begin{proof}
In \cite[Section~6.4]{CohenDeVoreNochetto:12} Cohen, DeVore and Nochetto construct some function $u_\lff\in H_0^1(\Omega)$ and a sequence $L_n$ as claimed for which \eqref{eq:CohenDeVoreNochetto-err} and 
\begin{align}
\label{eq:CohenDeVoreNochetto-elmres}
 \min_{\#\grid\le n} \Bigg(
  \sum_{z\in\vertices(\grid)} \norm[H^{-1}(\omega_z)]{f}^2
 \Bigg)^{1/2}
 \ge
  L_n\, n^{-1/2}
\end{align}
hold. It thus remains to establish \eqref{eq:CohenDeVoreNochetto-osc}.
To this end, we fix temporarily an arbitrary vertex $z\in\vertices$ of a conforming mesh $\grid$ and let $g\in\P_0(\grid)$. The inverse triangle and \eqref{res-err-loc} yield
\begin{align*}
 \norm[H^{-1}(\omega_z)]{\lff-g}
 &\ge
 \norm[H^{-1}(\omega_z)]{\Delta U_{\lff;\grid} + g}
  - \norm[H^{-1}(\omega_z)]{\lff + \Delta  U_{\lff;\grid}}
\\
 &\ge
 \norm[H^{-1}(\omega_z)]{\Delta U_{\lff;\grid} + g}
  - \norm[\omega_z]{\nabla (u_\lff-U_{\lff;\grid})}.
\end{align*}
By Lemma~\ref{L:stable-biort}, we have, for all $\elm\in\grid$,
\begin{align*}
 \dual{\Delta U_{\lff;\grid}}{\psi_\elm}
 &=
 \sum_{\side\in\sides}
  J(U_{\lff;\grid})|_\side \int_\side\chi_\side\psi_\elm\ds
  = 0
\intertext{and, for all $\side\in\sides$ and $\elm_1,\elm_2\in\grid$ with
	$\elm_1\cap\elm_2=\side$,}
\begin{split} 
 \dual{\Delta U_{\lff;\grid} + g}{\psi_\side}
 &=
 \int_\side J( U_{\lff;\grid})\psi_\side\ds
  + \sum_{i=1,2} g|_{\elm_i}\int_{\elm_i}\chi_{K_i}\psi_\side\dx
      \\
      &=\dual{\Delta U_{\lff;\grid}}{\psi_\side}.
\end{split}
\end{align*}
Theorem~\ref{T:Quantifying_local_dual_norms} therefore implies
\begin{align*}
 \norm[H^{-1}(\omega_z)]{\Delta U_{\lff;\grid} + g}
 &\Cgeq
 \sum_{i\in\mathcal{I}_z\cap\sides}
  \abs{\dual{\Delta U_{\lff;\grid} + g}{\frac{\psi_i}{\|\nabla\psi_i\|}}}
\\
 &=
 \sum_{i\in\mathcal{I}_z\cap\sides}
  \abs{ \dual{\Delta U_{\lff;\grid}}{\frac{\psi_i}{\|\nabla\psi_i\|}}}
\\
 &=
 \sum_{i\in\mathcal{I}_z}
  \abs{\dual{\Delta U_{\lff;\grid}}{\frac{\psi_i}{\|\nabla\psi_i\|}}}
 \Cgeq
 \norm[H^{-1}(\omega_z)]{\Delta U_{\lff;\grid}}.
\end{align*}
Exploiting also Lemma~\ref{L:localization}, we arrive at
\begin{multline*}
 \left(
  \sum_{z\in\vertices}
  \norm[H^{-1}(\omega_z)]{\Delta U_{\lff;\grid} + g}^2
 \right)^{1/2}
\\ \begin{aligned}
  &\Cgeq
  \left( \sum_{z\in\vertices}
   \norm[H^{-1}(\omega_z)]{\Delta U_{\lff;\grid}}^2
  \right)^{1/2}
\\
  &\Cgeq
  \left( \sum_{z\in\vertices}
   \norm[H^{-1}(\omega_z)]{f}^2
  \right)^{1/2}
  -
  \left( \sum_{z\in\vertices}
   \norm[H^{-1}(\omega_z)]{f+\Delta U_{\lff;\grid}}^2
  \right)^{1/2}
\\
  &\geq
  \left(
   \sum_{z\in\vertices} \norm[H^{-1}(\omega_z)]{f}^2
  \right)^{1/2}
  - C \, \norm[]{\nabla (u_\lff-U_{f,\grid})}.
\end{aligned}
\end{multline*}
Consequently, \eqref{eq:CohenDeVoreNochetto-err} and \eqref{eq:CohenDeVoreNochetto-elmres} lead to
\begin{align*}
 \min_{\#\grid\le n} \min_{g\in\P_0(\grid)} \left(
  \sum_{z\in\vertices} \norm[H^{-1}(\omega_z)]{\lff - g}^2
  \right)^{1/2}
  \geq 
  (L_n-C)\,n^{-1/2},
\end{align*}
which, upon redefining $(L_n)_n$, implies \eqref{eq:CohenDeVoreNochetto-osc} and the proof is finished.  
\end{proof}

\begin{rem}[Overestimation of $H^{-1}$-variant of standard residual estimator]
\label{R:over-H-1-std-res-est}
\ As pointed out by Cohen, DeVore, and Nochetto~\cite{CohenDeVoreNochetto:12}, the example of Lemma \ref{L:CDVD} entails that the right-hand side of
\begin{equation*}
\norm[]{\nabla(u_\lff - U_{{\lff};\grid})}^2
\Cleq
\sum_{z\in\vertices(\grid)}\norm[H^{-1}(\omega_z)]{\Delta U_{\lff;\grid}}^2+\norm[H^{-1}(\omega_z)]{f}^2,
\end{equation*}
a variant of the standard residual estimator defined for all loads $f \in H^{-1}(\Omega)$, is overestimating. In \S\ref{sec:residual} below, we propose through our new approach another variant that is free of overestimation.
\end{rem}

\section{Realizations with classical techniques}
\label{sec:comp_sur}
The a~posteriori error bounds in \S\ref{s:apost-bnds} are abstract in
that they are given in terms of the local dual norms
$\norm[H^{-1}(\omega_z)]{\cdot}$, $z \in \vertices$, of the discretized
residual and the oscillation. For the norms 
$\norm[H^{-1}(\omega_z)]{\ProG f+\Delta U_{\lff;\grid}}$, $z\in\vertices$,
of the discretized residual, we required a
quantification in terms of finite information on the load and
provided a possible realization in Theorem~\ref{T:Quantifying_local_dual_norms}.
In this section we  discuss a selection of alternative realizations. All
realizations are motivated by classical approaches
to a~posteriori analysis and cover two explicit and two implicit techniques.
It is worth making the following observations:
\begin{itemize}
\item Hierarchical estimators and estimators based upon local problems
implicitly introduce a splitting of the residual like the one proposed
in \S\ref{S:towards-err-dom-oscillation}.
\item The overestimation of the standard residual estimator in Remark
\ref{R:over-H-1-std-res-est} can be cured with the help of the splitting
of the residual in \S\ref{S:towards-err-dom-oscillation}.
\item Employing different local dual norms, the approach of \S\ref{S:APost}
can be extended to estimators based on flux equilibration.
\item Each realization quantifies a local dual norm of the discretized
residual by a computable, equivalent norm. Both equivalence and
computability hinge on the finite-dimensional nature of the discretized
residual.
\end{itemize}

\subsection{An hierarchical estimator}
\label{sec:hierarchical}
Hierarchical estimators investigate the residual on an extension of the
given finite element space. While higher order extensions were used
originally, Bornemann, Erdmann, and Kornhuber show in \cite{BoErKo:96}
that an extension containing the functions
\begin{align}
\label{lambda}
 \lambda_\elm:=\prod_{z\in\vertices\cap\elm}\phi_z, \quad\elm\in\grid,
\qquad\text{and}\qquad
 \lambda_\side:=\prod_{z\in\vertices\cap\side}\phi_z, \quad\side\in\sides,
\end{align}
already ensures reliability for piecewise constant loads
$\lff \in \P_0(\grid)$. The indicators of a corresponding, `minimal' hierarchical estimator are given by
\begin{align*}
 \est_{\mathrm{H}}(\lff,\grid,i)
 \definedas
 \abs{\dual{\Res(\lff;\grid)}{\frac{\lambda_i}{\norm[]{\nabla\lambda_i}}}},
\quad
 i\in \mathcal{I}=\grid\cup\sides,
\end{align*}
and computable in terms of $U_{\lff;\grid}$ and the evaluations $\dual{\lff}{\lambda_i}$,
$i \in \mathcal{I}$. This definition implies the constant-free
local lower bounds
\begin{equation*}
 \est_{\mathrm{H}}(\lff,\grid,i)
 \leq
 \norm[H^{-1}(\operatorname{supp}\lambda_i)]{\Res(\lff;\grid)}
\end{equation*}
and therefore, cf.\ \eqref{abs_low_loc}, we have that,
for every $z\in\vertices$ and ${I}_z= \{i\in\mathcal{I} \mid i \ni z \}$, 
\begin{align}
\label{eq:E>Ehi}
 \left(
 \sum_{i\in\mathcal{I}_z}
  \est_{\mathrm{H}}(\lff,\grid,i)^2
 \right)^{1/2}
 \leq
 \sqrt{d+1} \, \norm[H^{-1}(\omega_z)]{\Res(\lff,\grid)},
\end{align}
which is a local counterpart of the global lower bound
in Veeser~\cite[Lemma~3.3]{Veeser:02}.

This estimator is very closely related to the
discretized residuals of \S\ref{S:OFD} and Theorem~\ref{T:Quantifying_local_dual_norms}. Indeed, if $\elm\in\grid$ and $\side\in\sides$, $\elm_1,\elm_2\in\grid$ such that $\side = \elm_1 \cap \elm_2$, we have
\begin{equation}
\label{psi-lambda}
 \psi_\elm = \frac{(2d+1)!}{d!|\elm|} \lambda_\elm
\quad\text{and}\quad
 \psi_\side
 =
 \frac{(2d-1)!}{(d-1)!|\side|} \left(
  \lambda_\side - (2d+1)\sum_{i=1}^2 \lambda_{\elm_i}
 \right).
\end{equation}
in view of \eqref{biorth-test-fcts}. Hence
$\operatorname{span}\{\psi_i \mid i\in\mathcal{I}\} = \operatorname{span}\{\lambda_i \mid i\in\mathcal{I}\}$
and Remark \ref{R:ProG-orthogonality} yields $\dual{f}{\lambda_i}=\dual{\ProG f}{\lambda_i}$, $i\in\mathcal{I}$,
and the indicators may be viewed also as evaluations of the discretized residual: for $i \in \mathcal{I}$,
\begin{equation*}
 \est_{\mathrm{H}}(\lff,\grid,i)
 =
 \abs{
  \dual{\ProG\lff+\Delta U_{f,\grid}}
   {\frac{\lambda_i}{\norm[]{\nabla\lambda_i}}}
 }.
\end{equation*}
As a consequence, we also have the following counterpart of \eqref{eq:E>Ehi}:
\begin{equation}
\label{eq:Ehi=Res}
 \left(
  \sum_{i\in\mathcal{I}_z}
   \est_{\mathrm{H}}(\lff,\grid,i)^2
 \right)^{1/2}
 \leq
 \sqrt{d+1} \, \norm[H^{-1}(\omega_z)]{\ProG\lff+\Delta U_{f,\grid}}.
\end{equation}
In order to prove the converse bound, we may proceed with the help of $\ProG^*$ as in \cite{Veeser:02}. However, having Theorem~\ref{T:Quantifying_local_dual_norms}
at our disposal, it is simpler to exploit \eqref{psi-lambda}. We immediately see
\begin{subequations}
\label{eq:E<Ehi}
\begin{equation}
 \est_{\mathrm{H}}(\lff,\grid,\elm)
 =
 \abs{
 	\dual{\ProG\lff+\Delta U_{f,\grid}}
 	{\frac{\psi_\elm}{\norm[]{\nabla\psi_\elm}}}
 }.
\end{equation}
Moreover, given $\side\in\sides$, $\elm_1,\elm_2\in\grid$ with $\side = \elm_1 \cap \elm_2$, we deduce
\begin{equation*}
 C_d |\side|^{-1}
 \leq
 \max_\side \psi_\side
 \leq
 h_\side \max_{\elm_1} | \nabla \psi_\side |
 \Cleq
 h_\side |\elm|^{-1/2} \norm[\elm_1]{\nabla \psi_\side}
\end{equation*}
with $h_\side \definedas \diam\side$ and, for $i \in \{\side, \elm_1, \elm_2 \}$
\begin{equation*}
 \norm[\omega_\side]{ \nabla\lambda_i }
 \leq
 C_d \max_{i=1,2} \rho_{\elm}^{-1} |\omega_\side|^{1/2}.
\end{equation*}
We therefore obtain $\norm[]{\nabla \psi_\side}^{-1} \norm[]{\nabla\lambda_i} \Cleq |\side|$ and
\begin{equation}
 \abs{
 	\dual{\ProG\lff+\Delta U_{f,\grid}}
 	{\frac{\psi_\side}{\norm[]{\nabla \psi_\side}}}
 }
 \Cleq
 \sum_{i\in\{\side,\elm_1,\elm_2\}}
  \est_{\mathrm{H}}(\lff,\grid,i)
\end{equation}
\end{subequations}
Summing up, the hierarchical estimator quantifies the local discretized residual,
\begin{equation*}
 \sum_{i\in\mathcal{I}_z}
  \est_{\mathrm{H}}(\lff,\grid,i)^2
 \simeq
 \norm[H^{-1}(\omega_z)]{\ProG\lff+\Delta U_{f,\grid}},
\quad
 z\in\vertices,
\end{equation*}
and we have the following a posteriori bounds.

\begin{thm}[Hierarchical estimator with error-dominated oscillation]
For any functional $\lff \in H^{-1}(\Omega)$ and any conforming mesh
$\grid$, we have the global equivalence
\begin{align*}
	\norm[]{\nabla(u_\lff - U_{\lff;\grid})}^2
	\eqsim
	\sum_{i\in\mathcal{I}}
     \est_{\mathrm{H}}(\lff,\grid,i)^2
	+
	\sum_{z\in\vertices} 
	\norm[H^{-1}(\omega_z)]{\ProG\lff-\lff}^2,
\end{align*}
as well as the following local lower bounds: for every $z \in \vertices$, 
\begin{gather*}
 \sum_{i\in\mathcal{I}_z}
  \est_{\mathrm{H}}(\lff,\grid,i)^2
 \leq
 (d+1) \norm[\omega_z]{\nabla(u_\lff - U_{\lff;\grid})}^2,
\\
 \sum_{i\in\mathcal{I}_z}
  \norm[H^{-1}(\omega_z)]{\ProG\lff-\lff}^2
 \Cleq
	\norm[\omega_z]{\nabla(u_\lff - U_{\lff;\grid})}^2. 
\end{gather*}
The hidden constants depend only on $d$ and $\sigma(\grid)$. 
\end{thm}

\begin{proof}
Combine Theorem~\ref{T:ubd}, Theorem~\ref{T:lbd}, Corollary~\ref{C:lbd}, \eqref{res-err-loc}, \eqref{eq:E>Ehi}, and \eqref{eq:E<Ehi}.
\end{proof}

\subsection{An improved standard residual estimator}
\label{sec:residual}
The standard residual estimator applies suitably scaled norms
to the jump and element residual; see, e.g., Verf\"urth~\cite[Section~1.4]{Verfuerth:2013}.
In the case of the discretized residual
\begin{equation*}
 \ProG\lff + \Delta U_{\lff,\grid}
 =
 \sum_{\side\in\sides}
 \big( \dual{\lff}{\psi_\side}+J( U_{\lff;\grid})|_{\side} \big) \chi_\side 
 +
 \sum_{\elm\in\grid} \dual{\lff}{\psi_\elm} \chi_\elm,
\end{equation*}
this leads to the following indicators:
\begin{align*}
 \est_{\mathrm{R}}(U_{\lff;\grid},\ProG\lff,\side)
 &\definedas
 h_\side^{1/2} \norm[\side]{\dual{\lff}{\psi_\side}+J( U_{\lff;\grid})}
 &&\side\in\sides,
\\
 \est_{\mathrm{R}}(U_{\lff;\grid},\ProG\lff,\elm)
 &\definedas
 h_\elm \norm[\elm]{\dual{\lff}{\psi_\elm}},
 &&\elm\in\grid,
\end{align*}
where $h_\side$ and $h_\elm$ denote, respectively, the diameters of $\side$ and $\elm$
and computability is given in terms of $U_{\lff;\grid}$ and \eqref{avail_info}.

These indicators actually quantify the discretized residual and in a way that is very
tight to Theorem~\ref{T:Quantifying_local_dual_norms}: for any interelement face
$\side\in\sides$,
\begin{subequations}
\label{res-psi}
\begin{equation}
\label{res-psi_side}
 \est_{\mathrm{R}}(U_{\lff;\grid},\ProG\lff,\side)
 \eqsim
 \abs{\dual{\ProG \lff+\Delta
 		U_{\lff;\grid}}{\frac{\psi_\side}{\norm[]{\nabla\psi_\side}}} }
\end{equation}
and, for any element $\elm\in\grid$,
\begin{equation}
\label{res-psi_elm}
 \est_{\mathrm{R}}(U_{\lff;\grid},\ProG\lff,\elm)
\eqsim
\abs{\dual{\ProG \lff+\Delta
		U_{\lff;\grid}}{\frac{\psi_\elm}{\norm[]{\nabla\psi_\elm}}} },
\end{equation}
\end{subequations}
where the hidden constants depend only on $d$ and $\sigma(\grid)$. 
To see \eqref{res-psi_side}, let $\side\in\sides$ be any interelement face. Lemma~\ref{L:stable-biort}~\ref{L:stable-biort_algebra}, the trace
inequality~\eqref{trace-inequality} for $w=\psi_\side^2$ and the
Friedrichs inequality~\eqref{Friedrichs} for $v=\psi_\side$, both with
$\omega_\side$ in place of $\omega_z$, give
\begin{align*}
 &\abs{
 \dual{\ProG \lff + \Delta U_{\lff;\grid}}
  {\frac{\psi_\side}{\norm[]{\nabla\psi_\side}}}
 }
 =
 \abs{
  \dual{\big(\dual{\lff}{\psi_\side} + J( U_{\lff;\grid})|_\side 
	\big)\,\chi_\side}
  {\frac{\psi_\side}{\norm[]{\nabla\psi_\side}}}
 }
\\
 &\qquad\leq
 \norm[\side]{\dual{\lff}{\psi_\side}+J( U_{\lff;\grid})} \,
  \frac{\norm[\side]{\psi_\side}}{\norm[]{\nabla\psi_\side}}
 \leq
 h_\side^{1/2}\norm[\side]{\dual{\lff}{\psi_\side}+J( U_{\lff;\grid})},
\end{align*}
while \eqref{eq:DpsiF} yields $\norm{\nabla\psi_\side} \Cleq (h_\side |\side|)^{-1/2}$
and so
\begin{align*}
\abs{\dual{\ProG \lff+\Delta
		U_{\lff;\grid}}{\frac{\psi_\side}{\norm[]{\nabla\psi_\side}}} }
&=\frac{\norm[\side]{ \dual{\lff}{\psi_\side}+J( U_{\lff;\grid})}
}{\abs{\side}^{1/2}\norm[]{\nabla\psi_\side}}
\\
&\gtrsim
 h_\side^{1/2} \norm[\side]{\dual{\lff}{\psi_\side}+J( U_{\lff;\grid})}.  
\end{align*}
Similarly, we obtain \eqref{res-psi_elm}. 

Inserting the combination of Theorem~\ref{T:Quantifying_local_dual_norms}
and \eqref{res-psi} in the abstract a~posteriori analysis of
\S\ref{s:apost-bnds}, we obtain the following result.
\begin{thm}[Standard residual estimator with error-dominated oscillation]
\label{T:std-res-est}
For\\ any functional $\lff \in H^{-1}(\Omega)$ and any conforming mesh
$\grid$, we have the global equivalence
\begin{align*}
 \norm[]{\nabla(u_\lff - U_{\lff;\grid})}^2
 \eqsim
 \sum_{i\in\mathcal{I}} \est_{\mathrm{R}}(U_{\lff;\grid},\ProG \lff,i)^2
 +
 \sum_{z\in\vertices} \norm[H^{-1}(\omega_z)]{\ProG\lff-\lff}^2,
\end{align*}
as well as the following local lower bounds: for $z\in\vertices$, 
\begin{align*}
  \sum_{i\in\mathcal{I}_z}\est_{\mathrm{R}}(U_{\lff;\grid},\ProG\lff,i)^2
  +
  \norm[H^{-1}(\omega_z)]{\ProG\lff-\lff}^2
 \Cleq 
 \norm[\omega_z]{\nabla(u_\lff - U_{\lff;\grid})}^2. 
\end{align*}
The hidden constants depend only on $d$ and $\sigma(\grid)$. 
\end{thm}

Theorem~\ref{T:std-res-est} relies on key features of
the approach in \S\ref{S:APost}, which the following
remark elaborates on.

\begin{rem}[Classical vs new standard residual estimator]
In contrast to the classical standard residual estimator \eqref{d:estR}
and its $H^{-1}$-variant in Remark~\ref{R:over-H-1-std-res-est}, the
variant of Theorem~\ref{T:std-res-est} is completely equivalent to the error.
The reason for this improvement lies in a suitable correction of the original
jump residual. To elucidate this, remember that both the classical standard residual
estimator and its $H^{-1}$-variant in Remark~\ref{R:over-H-1-std-res-est}
do not discretize the residual and therefore compare them to
\begin{equation*}
 \sum_{\side\in\sides}
  h_\side^{1/2} \norm[\side]{J(U_{\lff;\grid}) + \dual{\lff}{\psi_\side}}^2
 +
 \sum_{z\in\vertices}
  \Big\|
   \lff -\sum_{\side\in\sides} \dual{\lff}{\psi_\side}\chi_\side
  \Big\|^2_{H^{-1}(\omega_z)},
\end{equation*} 
which also does not split off an infinite-dimensional part of the load $\lff$.
The corrections $\dual{\lff}{\psi_\side}$, $\side\in\sides$, of the
jump residual make sure that the new jump residual has the invariance
properties necessary for avoiding overestimation, \ie it vanishes whenever
the exact solution happens to be discrete.
Corrections with this property have been used previously. For example,
Nochetto~\cite{Nochetto:95} considers the special case
$\lff = \lff_1 +\divo \vec{f_2}$, where $f_1,\vec{f_2}$ are suitable functions,
and assigns $(\divo \vec{f_2})|_\elm$, $\elm\in\grid$, to the element residual
and the jumps in the normal trace of $\vec{f_2}$ across interelement sides
correct the jump residual. Similarly, in standard
residual estimators for the Stokes problem, pressure jumps correct the
jump residual associated with the velocity. The novelty is that the
corrections $\dual{\lff}{\psi_\side}$, $\side\in\sides$, are defined
for an arbitrary $\lff \in H^{-1}(\Omega)$ and also
locally $H^{-1}$-stable and so fulfill the second necessary condition
to avoid local overestimation.
Notably, the latter entails that, even if $\lff$ is a
smooth function, the jump residual will be corrected significantly in
certain cases.
%
%
\end{rem}

\subsection{An estimator based on local problems}
\label{sec:localProbs}
%
A local problem lifts the residual to a local extension of the given
finite element space and so provides a local correction,
the norm of which is used an error indicator; cf.\ Babu\v{s}ka and
Rheinboldt~\cite{BabuskaRheinboldt:78}. While computability requires finite-dimensional
extensions, the higher cost with respect to the previous explicit estimators
is tied up with the hope of improved accuracy.

The following instance from
Verf\"urth~\cite[Section~1.7.1 and Remark~1.21]{Verfuerth:2013} is vertex-based and
uses the local extensions 
\begin{align*}
 \U_z
 \definedas
 \operatorname{span}\{\lambda_i \mid i\in \mathcal{I}_z\}
 =
 \operatorname{span}\{\psi_i \mid i\in \mathcal{I}_z\},
\quad
 z\in\vertices,
\end{align*}
where the functions $\psi_i$ and $\lambda_i$ are defined, respectively, in
\eqref{biorth-test-fcts} and \eqref{lambda}. Given a vertex $z\in\vertices$,
the indicator is then
\begin{align*}
 \est_{\mathrm{L}}(\lff,\grid,z)
 \definedas
 \norm[]{\nabla \nu_z}, 
\end{align*}
where
\begin{align*}
 \nu_z\in\U_z
 \quad\text{such that}\quad
 \forall \lambda \in\U_z \quad
 \int_\Omega \nabla\nu_z\cdot\nabla \lambda \dx
 =
 \dual{ \Res(\lff;\grid) }{\lambda}. 
\end{align*}
Thus, $\nu_z$ is computable in terms of $U_{\lff;\grid}$ and, e.g.,
\eqref{avail_info}. The indicator $\est_{\mathrm{L}}(\lff,\grid,z)$ may
be viewed as an implicit counterpart of $(\sum_{i \in \mathcal{I}_z} \est_{\mathrm{H}}(\lff,\grid,i)^2)^{1/2}$ from \S\ref{sec:hierarchical}.
Taking $\lambda=\nu_z$, we immediately obtain the constant-free lower bound
\begin{equation}
\label{c-free-lbd}
 \est_{\mathrm{L}}(\lff,\grid,z)
 \leq
 \norm[H^{-1}(\omega_z)]{\Res(\lff;\grid)},
\end{equation}
which slightly improves upon \eqref{eq:E>Ehi}.

Notice that, in light of Remark \ref{R:ProG-orthogonality}, the solution
$\nu_z$ can be interpreted also as a lift of the discretized residual
$\ProG\lff + \Delta U_{\lff;\grid}$.  Consequently, the first inequality
in
\begin{align}
\label{loc-dis-res}
 \est_{\mathrm{L}}(\lff,\grid,z)
 \leq
 \norm[H^{-1}(\omega_z)]{\ProG\lff+\Delta U_{\lff;\grid}}
 \Cleq
 \est_{\mathrm{L}}(\lff,\grid,z)
\end{align}
is correct. The second one follows from Remark~\ref{R:adjoint}
and Theorem~\ref{T:loc_stability} in the spirit of Morin, Nochetto and
Siebert~\cite{MoNoSi:03}.
In fact, for $v\in H^1_0(\omega_z)$, we have
\begin{align*}
 \dual{\ProG \lff+\Delta U_{\lff;\grid}}{v}
 &=
 \dual{\Res(\lff;\grid)}{\ProG^*v}
 =
 \int_{\omega_z} \nabla\nu_z\cdot\nabla\ProG v \dx
\\
 &\leq
 \norm[]{\nabla\nu_z} \norm[\omega_z]{\nabla\ProG^* v}
 \Cleq
 \est_{\mathrm{L}}(\lff,\grid,z) \norm[\omega_z]{\nabla v}.
\end{align*}

%
\begin{thm}[Estimator based on local problems with error-dominated oscillation]
For any functional $\lff \in H^{-1}(\Omega)$ and any conforming mesh
$\grid$, we have the global equivalence
\begin{align*}
 \norm[]{\nabla(u_\lff - U_{\lff;\grid})}^2
 \eqsim
 \sum_{z\in\vertices}
  \est_{\mathrm{L}}(\lff,\grid,z)^2
  +
  \norm[H^{-1}(\omega_z)]{\ProG\lff-\lff}^2,
\end{align*}
as well as the following local lower bounds: for every $z\in\vertices$, 
\begin{gather*}
 \est_{\mathrm{L}}(\lff,\grid,z)
 \leq
 \norm[\omega_z]{\nabla(u_\lff - U_{\lff;\grid})}
\;\text{ and }\;
 \norm[H^{-1}(\omega_z)]{\ProG\lff-\lff}
 \Cleq
 \norm[\omega_z]{\nabla(u_\lff - U_{\lff;\grid})}. 
\end{gather*}
The hidden constants depend only on $d$ and $\sigma(\grid)$. 
\end{thm}

\begin{proof}
Combine Theorem~\ref{T:ubd}, Theorem~\ref{T:lbd}, Corollary~\ref{C:lbd}, \eqref{res-err-loc}, \eqref{c-free-lbd} and \eqref{loc-dis-res}.
\end{proof}

\subsection{An estimator based on flux equilibration}
\label{sec:PragerSynge}
%
While indicators based on local problems provide constant-free local
lower bounds, estimators based on flux equilibration aim for a
constant-free, or at least explicit, global upper bound. This is achieved
with the help of other, more sophisticated
liftings within the framework of the fundamental theorem of Prager and
Synge~\cite{PragerSynge:47}, which, for the homogeneous Dirichlet problem
\eqref{eq:elliptic}, can be formulated as follows: For any $v\in H_0^1(\Omega)$,
we have 
\begin{align}\label{eq:PragerSynge}
\norm[]{\nabla(v-u) }=\min\left\{\norm[]{
	\vec{\xi}}\mid
\vec\xi\in L^2(\Omega;\R^d)~\text{with}~
\divo\vec{\xi} =\Delta v+\lff~\text{in}~
H^{-1}(\Omega)\right\}.
\end{align}
Realizations of this idea in
Ainsworth~\cite{Ainsworth:2010},
Braess and Sch\"oberl~\cite{BraessSchoeberl:2008},
Ern, Smears and Vohralik~\cite{ErnSmearsVohralik:17a,ErnVohralik:15}, and
Luce and Wohlmuth~\cite{LuceWohlmuth:04}
make use of some classical oscillation. Its replacement by an
error-dominated oscillation requires some adjustment to the
approach of \S\ref{S:APost}.


The upper bound in the localization of Lemma \ref{L:localization} involves a
non-explicit multiplicative constant. In order to improve on this, we replace
the local spaces $H^1_0(\omega_z)$, $z\in\vertices$, with
\begin{equation*}
 H_z
 \definedas
 \begin{cases}
  \{v \in H^1(\omega_z) \mid \int_{\omega_z}v=0\}, &\text{if }z\in\vertices_0 = \vertices \cap \Omega,
 \\
  \{v\in H^1(\omega_z) \mid v|_{\partial\omega_z\cap\partial\Omega}=0\},
  &\text{if } z \in \vertices\setminus\vertices_0,
 \end{cases}
\end{equation*}
equip them with the norm $\norm[\omega_z]{\nabla\cdot}$, and denote the respective dual
spaces by $H_z^*$.
  
%
\begin{lem}[Alternative localization]
\label{L:localization2}
Let $\lfg \in H^{-1}(\Omega)$ be any functional. 
\begin{enumerate}[label=(\roman*)]
\item\label{L:loc2i} If $\lfg\in\mathcal{R}_\grid$, then
\begin{equation*}
		\norm[H^{-1}(\Omega)]{\lfg}^2
		\le
		(d+1)\sum_{z\in\vertices}  \norm[H^{*}_z]{\phi_z\lfg}^2.
\end{equation*}
\item\label{L:loc2ii} We have
\begin{equation*}
		\sum_{z\in\vertices} \norm[H^{*}_z]{\phi_z\lfg}^2
		\Cleq 
		\norm[H^{-1}(\Omega)]{\lfg}^2,
\end{equation*}
where the hidden constant 
depends only on $d$ and the shape coefficient $\sigma(\grid)$. 
\end{enumerate}
\end{lem}

\begin{proof}
The proof is essentially a regrouping of the arguments proving Lemma~\ref{L:localization}, where \eqref{eq:vphiz} slips into the proof of \ref{L:loc2ii}; cf.\ Canuto et al.~\cite[Proposition~3.1]{CanutoNochettoStevensonVerani:2017}. 
\end{proof}


Splitting the residual up in discretized residual and oscillation, we then obtain
the following abstract error bounds; we do not state the global lower bound as it
is immediate consequence of the local one.
\begin{lem}[Alternative abstract error bounds]
\label{L:PragerSynge}
For any functional $\lff \in H^{-1}(\Omega)$ and any conforming mesh
$\grid$, we have the global upper bound
\begin{align*}
	\norm[]{\nabla(u_\lff - U_{\lff;\grid})}^2
	\leq 
	\sum_{z\in\vertices}\left(\norm[H_z^*]{\phi_z(\ProG\lff+\Delta U_{\lff;\grid})}+
	\norm[H_z^*]{\phi_z(\ProG\lff-\lff)}\right)^2,
\end{align*}
as well as the following local lower bounds: for every vertex $z\in\vertices$, 
\begin{align*}
	\norm[H_z^*]{\phi_z(\ProG\lff+\Delta U_{\lff;\grid})}
	+ \norm[H_z^*]{\phi_z(\ProG\lff-\lff)}
	\Cleq
	\norm[\omega_z]{\nabla(u_\lff - U_{\lff;\grid})}
\end{align*}
The hidden constants depend only on $d$ and $\sigma(\grid)$. 
\end{lem}

\begin{proof}
The global upper bound follows from Lemma~\ref{L:localization2}~\ref{L:locii}
and the triangle inequality. To prove the local lower bounds, we recall
Theorem~\ref{T:lbd} and take $\ell =\ProG\lff+\Delta U_{\lff;\grid}$ and
$\ell= \ProG\lff-\lff$ in 
\begin{align}\label{eq:loc2<loc}
	\dual{\phi_z\lfg}{v_z}=\dual{\lfg}{\phi_zv_z}\le
	\norm[H^{-1}(\omega_z)]{\lfg} \norm[\omega_z]{\nabla(v_z\phi_z)}
	\Cleq \norm[H^{-1}(\omega_z)]{\lfg}\norm[\omega_z]{\nabla v_z},
\end{align}
which exploits~\eqref{eq:vphiz} for $v_z\in H_z$ and $z\in\vertices$.
\end{proof}

In order to quantify the local discretized residual, we construct local
equilibrated fluxes following the ideas of Braess, Pillwein, and
Sch\"oberl~\cite{BrPiSch:2009} and Ern, Smears, and
Vohral\'{i}k~\cite{ErnSmearsVohralik:17a}.
To this end, fix any vertex $z\in\vertices$ and define the operator
$\pi_z:\{ \phi_z v \mid v \in H^1(\omega_z) \} \to H_z^*$ by 
\begin{align}
\label{df:piz}
 \pi_z \big( \phi_z\lfg \big)
 :=
 \begin{cases}  
  \phi_z\lfg - \frac{\dual{\phi_z\lfg}{1}}{|\omega_z|} &\text{if}~z\in\vertices_0,
\\
  \phi_z\ell &\text{if }z\in\vertices\setminus\vertices_0.
 \end{cases}
\end{align}
We emphasize that $\pi_z\big(\phi_z(\ProG\lff+\Delta U_{\lff;\grid})\big)$ can be
computed in terms of $U_{\lff;\grid}$ and \eqref{avail_info}. Thanks to the
definition of the spaces $H_z$, $z\in\vertices$,
and the general form of the theorem of Prager and Synge (see, e.g., Verf\"urth \cite[Proposition~1.40]{Verfuerth:2013}), we have 
\begin{align}
\label{eq:duality}
 \norm[H_z^*]{\phi_z(\ProG\lff+\Delta U_{\lff;\grid})}
 =
 \norm[H_z^*]{\pi_z\phi_z(\ProG\lff+\Delta U_{\lff;\grid})}
 =
 \min_{\vec{\xi}\in \W_z}\norm[\omega_z]{\vec{\xi}}
\end{align}
with the affine space
\begin{align*}
 \W_z
 :=
 \big\{ \vec{\xi}\in L^2(\omega_z;\R^d) \mid
  &\divo\vec{\xi}
  =
  \pi_z\big(\phi_z(\ProG\lff+\Delta U_{\lff;\grid})\big)\in H_z^*
\\
 &\text{and}~\vec{\xi}\cdot\normal
 =
 0~\text{on}~\partial\omega_z~\text{if}~z\in\vertices_0
\\
 &\text{and}~\vec{\xi}\cdot\normal
 =
 0~\text{on}~\partial\omega_z\setminus\partial\Omega ~\text{if}~z\in\vertices\setminus\vertices_0
 \big\},
\end{align*}
and the equalities in the definition of $\W_z$  have to be understood in the sense of
distributions; the space $\W_z$ is not empty since 
$\dual{\pi_z\big(\phi_z(\ProG\lff+\Delta U_{\lff;\grid})\big)}{1}=0$ for every $z\in\vertices_0$.

In order to introduce a discrete counterpart of $\W_z$ in \eqref{eq:duality}, we employ
the Raviart-Thomas-N\'ed\'elec spaces
\begin{align*}
 \RTN(\elm)
 :=
 \{ \vec{\Xi} : {\elm}\to\R^d \mid
  \vec{\Xi}(x)=\vec{a}+b x~\text{for some}~ \vec{a}\in\P_1^d, b\in\P_1
 \},
\quad
 \elm\in\grid,
\end{align*}
and define 
\begin{align*}
 \W_z(\grid)
 \definedas
 \big\{ \vec{\Xi}\in L^2(\omega_z) \mid\
  &\vec{\Xi}|_\elm\in\RTN(\elm) \text{ for all }\elm\in\grid~\text{with}~\elm\subset \omega_z
\\
 &\text{and}~\divo\vec{\Xi} = \pi_z\big(\phi_z(\ProG\lff+\Delta U_{\lff;\grid}) \big)\in H_z^*
\\
 &\text{and}~\vec{\Xi}\cdot\normal = 0~\text{on}~\partial\omega_z~\text{if}~z\in\vertices_0
\\
 &\text{and}~\vec{\Xi}\cdot\normal = 0~\text{on}~\partial\omega_z\setminus\partial\Omega
~\text{if}~z\in\vertices\setminus\vertices_0\big\},
\end{align*}
which satisfies
\begin{align}\label{eq:Xi<xi}
 \min_{\vec{\Xi}\in\W_z(\grid)} \norm[\omega_z]{\vec{\Xi}}
 \Cleq
 \min_{\vec{\xi}\in\W_z} \norm[\omega_z]{\vec{\xi}}
 \le
 \min_{\vec{\Xi}\in\W_z(\grid)} \norm[\omega_z]{\vec{\Xi}}
\end{align}
and the hidden constant depends only on $d$ and $\sigma(\grid)$.
Indeed, the right inequality is obvious because of 
$\W_z(\grid)\subset\W_z$. 
The left inequality can be proved
by an explicit construction; see, e.g.,
\cite{BrPiSch:2009,ErnSmearsVohralik:17a}.
For the ease of presentation, however, we shall assume
\begin{align*}
 \vec{\Xi}_z
 \definedas
 \underset{_{\vec{\Xi}\in\W_z(\grid)}}{\operatorname{arg\,min}}
  \norm[\omega_z]{\vec{\Xi}} 
\end{align*}
and note
\begin{equation*}
 \norm[\omega_z]{\vec{\Xi}_z}
 \Cleq
 \norm[H_z^*]{\phi_z(\ProG\lff+\Delta U_{\lff;\grid})}
 \leq
 \norm[\omega_z]{\vec{\Xi}_z}.
\end{equation*}
in view of \eqref{eq:duality} and \eqref{eq:Xi<xi}.
Inserting this in the abstract bounds of Lemma~\ref{L:PragerSynge}, we readily
obtain the following a posteriori bounds; as before, we suppress the global
lower bound.

\begin{thm}[Equilibrated flux estimator with error-dominated oscillation]
\label{T:PragerSynge}
\hfill For \\ any functional $\lff \in H^{-1}(\Omega)$ and any conforming mesh
$\grid$, we have the global upper bound
\begin{align*}
	\norm[]{\nabla(u_\lff - U_{\lff;\grid})}^2
	\leq (d+1)
	\sum_{z\in\vertices}\left(\norm[\omega_z]{\vec{\Xi}_z}+
	\norm[H_z^*]{\phi_z(\ProG\lff-\lff)}\right)^2
\end{align*}
as well as the following local lower bounds: for every vertex $z\in\vertices$, 
\begin{align*}
	\norm[\omega_z]{\vec{\Xi}_z}^2+\norm[H_z^*]{\phi_z(\ProG\lff-\lff)}^2\Cleq
	\norm[\omega_z]{\nabla(u_\lff - U_{\lff;\grid})}^2. 
\end{align*}
The hidden constant depends only on $d$ and $\sigma(\grid)$. 
\end{thm}

In contrast to the cited previous bounds,
the upper bound in Theorem~\ref{T:PragerSynge} contains the multiplicative constant $d+1$. This constant arises from the localization in Lemma \ref{L:localization2}. As an alternative to this localization, one may use the constant-free upper bound in the following remark and split the estimator part $\norm[]{\vec{\Xi}}$ therein into local $L^2$-contributions.
\begin{rem}[Alternative upper bound]
Observing that
\begin{align*}
     \sum_{z\in\vertices}\divo\vec{\Xi}_z =
    \lff+\Delta
     U_{\lff;\grid} + \sum_{z\in\vertices}\pi_z\big(\phi_z(\ProG\lff-\lff)\big),
\end{align*}
we set $\vec{\Xi}\definedas \sum_{z\in\vertices}\vec{\Xi}_z$ and apply the theorem of Prager and Synge~\eqref{eq:PragerSynge} globally and Lemma \ref{L:localization2} to obtain
	\begin{align*}
	\norm[]{\nabla(u_\lff - U_{\lff;\grid})}
	\leq \norm[]{\vec{\Xi}}+ \sqrt{d+1}  \left(\sum_{z\in\vertices}\norm[H^*_z]{\phi_z(\ProG\lff-\lff)}^2\right)^{1/2}.
	\end{align*}
\end{rem}

\subsection*{Acknowledgment}
The authors would like to thank R\"udiger Verf\"urth for the proposal to define a precursor of the operator $\ProG$ in terms of a biorthogonal system.



\begin{thebibliography}{10}
	
	\bibitem{Ainsworth:2010}
	{\scshape M.~Ainsworth}, {\em A framework for obtaining guaranteed error bounds
		for finite element approximations}, Journal of Computational and Applied
	Mathematics, 234 (2010), pp.~2618--2632.
	
	\bibitem{AinsworthOden:00}
	{\scshape M.~Ainsworth and J.~T.~Oden}, {\em A Posteriori Error Estimation in
		Finite Element Analysis}, Pure and Applied Mathematics, Wiley-Interscience,
	New York, 2000.
	
	\bibitem{BabuskaMiller:87}
	{\scshape I.~Babu\v{s}ka and A.~Miller}, {\em A feedback finite element method
		with a~posteriori error estimation: {I}. {T}he finite element method and some
		basic properties of the a posteriori error estimator}, Comput. Methods Appl.
	Mech. Engrg., 61 (1987), pp.~1--40.
	
	\bibitem{BabuskaRheinboldt:78}
	{\scshape I.~Babu\v{s}ka and W.~Rheinboldt}, {\em Error estimates for adaptive
		finite element computations}, SIAM J. Numer. Anal., 15 (1978), pp.~736--754.
	
	\bibitem{BlechtaMalekVohralik:ta}
	{\scshape J.~Blechta, J.~M\'{a}lek, and M.~Vohral\'{i}k}, {\em Localization of
		the {$W^{-1,q}$}-norm for local a posteriori efficiency}.
	\newblock IMA J. Numer. Anal., to appear.
	
	\bibitem{BoErKo:96}
	{\scshape F.~A. Bornemann, B.~Erdmann, and R.~Kornhuber}, {\em A posteriori
		error estimates for elliptic problems in two and three space dimensions},
	SIAM J. Numer. Anal., 33 (1996), pp.~1188--1204.
	
	\bibitem{Braess:2007}
	{\scshape D.~Braess}, {\em Finite elements}, Cambridge University Press,
	Cambridge, third~ed., 2007.
	\newblock Theory, fast solvers, and applications in elasticity theory,
	Translated from the German by Larry L. Schumaker.
	
	\bibitem{BrPiSch:2009}
	{\scshape D.~Braess, V.~Pillwein, and J.~Sch{{\"o}}berl}, {\em Equilibrated
		residual error estimates are {$p$}-robust}, Comput. Methods Appl. Mech.
	Engrg., 198 (2009), pp.~1189--1197.
	
	\bibitem{BraessSchoeberl:2008}
	{\scshape D.~Braess and J.~Sch{{\"o}}berl}, {\em Equilibrated residual error
		estimator for edge elements}, Math. Comp., 77 (2008), pp.~651--672.
	
	\bibitem{CanutoNochettoStevensonVerani:2017}
	{\scshape C.~Canuto, R.~H. Nochetto, R.~Stevenson, and M.~Verani}, {\em On
		$p$-robust saturation for $hp$-afem}, Computers {\&} Mathematics with
	Applications, 73 (2017), pp.~2004--2022.
	
	\bibitem{CohenDeVoreNochetto:12}
	{\scshape A.~Cohen, R.~DeVore, and R.~H. Nochetto}, {\em Convergence rates of
		{AFEM} with {$H^{-1}$} data}, Found. Comput. Math., 12 (2012), pp.~671--718.
	
	\bibitem{ErnGuermond:2019}
	{\scshape A.~Ern and J.-L. Guermond}, {\em Finite elements: {II}.
		{A}pplications to {PDE}s.}
	\newblock in preparation.
	
	\bibitem{ErnSmearsVohralik:17a}
	{\scshape A.~Ern, I.~Smears, and M.~Vohral{\'\i}k}, {\em Discrete $p$-robust
		${H}(\textrm{div})$-liftings and a posteriori estimates for elliptic problems
		with ${H}^{-1}$ source terms}, Calcolo,  (2017), pp.~1--17.
	
	\bibitem{ErnVohralik:15}
	{\scshape A.~Ern and M.~Vohral{\'{\i}}k}, {\em Polynomial-degree-robust a
		posteriori estimates in a unified setting for conforming, nonconforming,
		discontinuous {G}alerkin, and mixed discretizations}, SIAM J. Numer. Anal.,
	53 (2015), pp.~1058--1081.
	
	\bibitem{Kreuzer.Veeser:ip}
	{\scshape C.~Kreuzer and A.~Veeser}, {\em A~posteriori analysis with
		error-dominated oscillation for higher order methods}.
	\newblock in preparation.
	
	\bibitem{LuceWohlmuth:04}
	{\scshape R.~Luce and B.~I. Wohlmuth}, {\em A local a posteriori error
		estimator based on equilibrated fluxes}, SIAM J. Numer. Anal., 42 (2004),
	pp.~1394--1414.
	
	\bibitem{MoNoSi:02}
	{\scshape P.~Morin, R.~H. Nochetto, and K.~G. Siebert}, {\em Convergence of
		adaptive finite element methods}, SIAM Review, 44 (2002), pp.~631--658.
	
	\bibitem{MoNoSi:03}
	\leavevmode\vrule height 2pt depth -1.6pt width 23pt, {\em Local problems on
		stars: {A} posteriori error estimators, convergence, and performance}, Math.
	Comp., 72 (2003), pp.~1067--1097.
	
	\bibitem{Nochetto:95}
	{\scshape R.~H. Nochetto}, {\em Pointwise a posteriori error estimates for
		elliptic problems on highly graded meshes}, Math. Comp., 64 (1995),
	pp.~1--22.
	
	\bibitem{NochettoVeeser:12}
	{\scshape R.~H. Nochetto and A.~Veeser}, {\em Primer of adaptive finite element
		methods}, in Multiscale and adaptivity: modeling, numerics and applications,
	vol.~2040 of Lecture Notes in Math., Springer, Heidelberg, 2012,
	pp.~125--225.
	
	\bibitem{PragerSynge:47}
	{\scshape W.~Prager and J.~L. Synge}, {\em Approximations in elasticity based
		on the concept of function space}, Quart. Appl. Math., 5 (1947),
	pp.~241--269.
	
	\bibitem{Stevenson:07}
	{\scshape R.~Stevenson}, {\em Optimality of a standard adaptive finite element
		method}, Found. Comput. Math., 7 (2007), pp.~245--269.
	
	\bibitem{Veeser:02}
	{\scshape A.~Veeser}, {\em Convergent adaptive finite elements for the
		nonlinear {L}aplacian}, Numer. Math., 92 (2002), pp.~743--770.
	
	\bibitem{Verfuerth:96}
	{\scshape R.~Verf\"urth}, {\em A Review of A Posteriori Error Estimation and
		Adaptive Mesh-Refinement Techniques}, Adv. Numer. Math., John Wiley,
	Chichester, UK, 1996.
	
	\bibitem{Verfuerth:2013}
	{\scshape R.~Verf{\"u}rth}, {\em A posteriori error estimation techniques for
		finite element methods}, Numerical Mathematics and Scientific Computation,
	Oxford University Press, Oxford, 2013.
	
\end{thebibliography}

\end{document}